\documentclass[11pt]{article}

\usepackage[utf8]{inputenc} 
\usepackage[T1]{fontenc}    
\usepackage{hyperref}       
\usepackage{booktabs}       
\usepackage{nicefrac}       
\usepackage{microtype}      
\usepackage{xcolor}         

\usepackage{pdflscape}
\usepackage{enumitem}
\usepackage{latexsym}
\usepackage{subcaption}
\usepackage{color,varioref}
\usepackage{longtable}
\usepackage{subfiles}
\usepackage{rotating,multirow,array}
\usepackage{booktabs}
\usepackage{afterpage} 
\usepackage{lscape} 

\usepackage{algorithm}
\usepackage{algorithmic}
\usepackage{amsfonts}
\usepackage{amsmath}
\usepackage{amsopn}
\usepackage{amssymb}
\usepackage{epstopdf}
\usepackage{graphicx}
\usepackage{lipsum}
\usepackage{mathrsfs}
\usepackage{url}
\usepackage{authblk}
\usepackage{amsthm}
\numberwithin{equation}{section}

\newtheorem{theorem}{Theorem}[section]
\newtheorem{lemma}[theorem]{Lemma}

\newtheorem{definition}{Definition}

\theoremstyle{remark}
\newtheorem{remark}{Remark}

\newcommand{\be}{\begin{equation}}
\newcommand{\ee}{\end{equation}}

\newcommand{\bm}{\boldsymbol}

\DeclareMathOperator*{\argmin}{\arg\min}

\def\calA{\mathcal{A}}
\def\calC{\mathcal{C}}
\def\calM{\mathcal{M}}
\def\calP{\mathcal{P}}

\def\calV{\mathcal{V}}

\def\Conv{\mbox{Conv}}

\def\LMO{\mbox{LMO}}
\def\SFW{\mbox{SFW}}
\def\rSFW{\mbox{rSFW}}
\def\SLMO{\mbox{SLMO}}

\def\bfzero{{\bm 0}}
\def\bfone{{\bm 1}}
\def\bfb{{\bm b}}
\def\bfc{{\bm c}}
\def\bfd{{\bm d}}
\def\bfe{{\bm e}}
\def\bfp{{\bm p}}
\def\bfg{{\bm g}}
\def\bfr{{\bm r}}
\def\bfu{{\bm u}}
\def\bfv{{\bm v}}
\def\bfx{{\bm x}}
\def\bfy{{\bm y}}
\def\bfz{{\bm z}}

\def\bflambda{\bm{\lambda}}

\begin{document}

\title{Simplex Frank-Wolfe: Linear Convergence and Its Numerical Efficiency for Convex Optimization over Polytopes 
}

\author[1]{Haoning Wang}
\author[2]{Houduo Qi \thanks{Corresponding author: 
houduo.qi@polyu.edu.hk}}
\author[1]{Liping Zhang}
\affil[1]{Department of Mathematical Sciences, Tsinghua University, Beijing 
100084, China}
\affil[2]{Department of Data Science and Artificial Intelligence, and Department of Applied Mathematics, The Hong Kong Polytechnic University, Hong Kong}

\date{}

\maketitle

\begin{abstract}
We investigate variants of the Frank-Wolfe (FW) algorithm for smoothing and strongly convex optimization over polyhedral sets, with the goal of designing algorithms that achieve linear convergence while minimizing per-iteration complexity as much as possible.
Starting from the simple yet fundamental unit simplex, and based on geometrically intuitive motivations, we introduce a novel oracle called Simplex Linear Minimization Oracle (SLMO), which can be implemented with the same complexity as the standard FW oracle.
We then present two FW variants based on SLMO: Simplex Frank-Wolfe and the refined Simplex Frank-Wolfe (rSFW).
Both variants achieve a linear convergence rate for all three common step-size rules.
Finally, we generalize the entire framework from the unit simplex to arbitrary polytopes.
Furthermore, the refinement step in rSFW can accommodate any existing FW strategies such as the well-known away-step and pairwise-step, leading to outstanding numerical performance.
We emphasize that the oracle used in our rSFW method requires only one more vector addition compared to the standard LMO, resulting in the lowest per-iteration computational overhead among all known Frank-Wolfe variants with linear convergence.

{\bf keywords:} Frank-Wolfe algorithm, conditional gradient methods, linear convergence, convex optimization, first-order methods, linear programming
\end{abstract}

\section{Introduction}

Over the past decades, Frank-Wolfe (FW) algorithms \cite{frank1956algorithm} (a.k.a. conditional gradients \cite{levitin1966constrained}) have been extensively investigated due to its lower per-iteration complexity compared to projected or proximal gradient-based methods, in particular for large-scale machine learning applications and sparse optimization.
This topic has been comprehensively covered in several recent publications including
\cite{bomze2021frank,braun2022conditional,pokutta2024frank} and
\cite[Chapter~7]{lan2020first},\cite[Chapter~10]{gartner2023optimization}, to just name a few.
The key step in FW algorithms is Linear Minimization Oracle (LMO).
We refer to \cite{lan2013complexity}
for (worst-case) complexity analysis for general LMOs.
One of the most often cited examples is LMO over the unit Simplex
$S_n := \left\{ \bfx \in \mathbb{R}^n \vert \ \sum x_i = 1, \bfx\ge 0  \right\}$.
Projection onto $S_n$ is much expensive than LMO over $S_n$.
Research effort has been on developing LMOs that may lead to linear convergence while
keeping the computation of each LMO as low as possible. Therefore, the total
computational complexity of a typical FW-type algorithm can be calculated as follows.
\[
  \mbox{Total Computation}
  = (\# \mbox{Iterations}) \times \mbox{(Computation of LMOs per iteration)}.
\]
Note that some existing algorithms may require more than one LMO each iteration.
The purpose of this paper is to propose a new LMO, whose computational complexity is probably the
cheapest among all existing algorithms. Furthermore, it also guarantees a linear convergence rate
comparable to the known ones for the convex optimization over a polytope:
\be \label{Problem: general}
	\min \ f(\bfx) \qquad
	\mbox{s.t.} \quad  \bfx\in\calP = \Conv(\mathcal{V}),
\ee
where $\mathcal{V}\subseteq \mathbb{R}^n$ is a \textit{finite} set of vectors that we call \textit{atoms}.
For the moment, we only assume $f: \calC \mapsto \mathbb{R}$ is convex and differentiable for the
convenience of discussion below.
Here, $\calC$ is an open set containing $\calP$. Later, we will enforce strong convexity as well
as other properties for our analysis.

\subsection{Related Work}\label{Section: Related_work}

There are a large number of publications that directly or remotely motivated this work.
We are only able to list a few of them below with some critical analysis.
Given an LMO, the original FW algorithm \cite{frank1956algorithm} states as:
\[
 \left\{
   \begin{array}{l}
   		 \bfy_k = \LMO(\nabla f(\bfx_{k-1}), \calP) \in \argmin\left\{
   		  \langle \nabla f(\bfx_{k-1}), \bfy \rangle \ \vert \ \bfy \in \calP
   		 \right\}, \\ [0.2ex]
   		 \bfx_k = (1-\delta_k)\bfx_{k-1} + \delta_k \bfy_k, \ \
   		 \delta_k \in (0, 1],
   	\end{array}
 \right .
\]
where $\delta_k$ is a steplength often satisfying certain conditions \cite{clarkson2010coresets,hazan2008sparse,jaggi2013revisiting}.
One of the key advantages of the FW method over the well-known projected gradient method is its lower cost per iteration in many common scenarios, such as the simplex \cite{clarkson2010coresets}, flow polytope \cite{combettes2021complexity,joulin2014colocalization}, spectrahedron \cite{hazan2008sparse,garber2016faster}, and nuclear norm ball \cite{jaggi2010simple}. This efficiency makes the FW method particularly advantageous for large-scale problems.
Numerous studies \cite{jaggi2013revisiting,lan2013complexity,freund2016new} have demonstrated that the convergence rate of the FW method is $\mathcal{O}(\frac{1}{k})$ and that this rate is generally not improvable,
except for some special cases, e.g., when the optimal solution lies in the interior of the constraint set \cite{guelat1986some}.
In fact, there exist examples for which the convergence rate of the FW method does not improve even when the objective function is strongly convex, see \cite{jaggi2013revisiting,lan2013complexity}.

Therefore, modifications on the original FW method must be made in order to achieve linear convergence
rate. Significant advances have been made alone this line of research and
there exist a large number of variants of FW methods that enjoy linear convergence rate, see
\cite[Chapter 3]{braun2022conditional}.
The well-known ones include FW-method with away-step (AFW)
 and the pairwise FW (PFW) \cite{lacoste2015global,damla2008linear,garber2013playing}.
Most of those modified methods can be cast in the following framework:
\be \label{Modified-FW}
\left\{
 \begin{array}{l}
 	 \bfy_k = \LMO(\nabla f(\bfx_{k-1}), \calP_k) \in \argmin\left\{
 	 \langle \nabla f(\bfx_{k-1}), \bfy \rangle \ \vert \ \bfy \in \calP_k
 	 \right\}, \\ [0.2ex]
 	 \bfg_k = \mbox{direction-correction satisfying certain conditions}, \\ [0.2ex]
 	 \bfx_k = (1-\delta_k)\bfx_{k-1} + \delta_k ( \bfy_k + \bfg_k), \ \
 	 \delta_k \in (0, 1],
 \end{array}
\right .
\ee
where $\calP_k \subseteq \calP$ is a well-constructed convex subset of $\calP$ at the current iterate $\bfx_{k-1}$.
This framework has a flexibility for more technical strategies to be added. For example,
one may mix $\bfy_k$ and $\bfg_k$ through certain combinations with some linesearch strategies.
Both AFW and PFW make use of such flexibility.
One major concern is that the computation of LMO over $\calP_k$ may be significantly higher than
that over $\calP$. 
This is the case when $\calP$ is
Simplex-like polytopes including $S_n$.

In a significant development aiming to address this issue, Garber and Hazan \cite{garber2016linearly}
proposed the methodology of LLOO (Local Linear Optimization Oracle), where $\calP_k$ is
the intersection of $S_n$ and a $\ell_1$-ball: 
\[
  \calP_k = S_n \cap B_1(\bfx_{k-1}, d_k), \quad \mbox{with}\quad
  B_1(\bfx, d) = \left\{
   \bfy \ | \ \| \bfx - \bfy \|_1 \le d
  \right\}.
\]
A key result is that LMO over this $\calP_k$ is LLOO, which is referred to as $\ell_1$-LMO.
Hence, linear convergence follows when
the radius is exponentially reduced at each iteration under the strongly convex setting.
We note that the framework of LLOO can be cast as a special case of
Shrinking Conditional Gradient Methods (sCndG) by Lan \cite[Eq.~3.34]{lan2013complexity} and \cite[Alg.~7.2]{lan2020first},
where an arbitrary norm is used.
The LLOO framework does not require
the step of $\bfg_k$ in \eqref{Modified-FW}.
To understand its actual performance, Fig.~\ref{Fig: LMO_compare_Simplex} in Sect.~\ref{Section-Numerical} illustrates its computational time
in comparison to the LMO over the unit simplex as well a projection algorithm.

It can be clearly observed that the time taken by
$\ell_1$-LMO is roughly same as the projection method, but is significantly slower (e.g.,$100\times$ slower when $n$ gets big) than $\LMO(\bfc, S_n)$.
There is a deep reason behind this performance and it can be best appreciated from
the perspective of geometric intuition by considering the situation of $n=3$.
Fig.~\ref{Fig: geometric_l1} illustrates the intersection of $\ell_1$-ball with the unit simplex.
Note that for any point $\bfx\in S_3$, the intersection of $\ell_1$-ball with the hyperplane containing $S_3$ forms a regular hexagon.
As the center $\bfx$ and radius $d$ vary, the shape corresponding to the intersection of this regular hexagon and unit simplex becomes more complex, as shown by the blue region in Fig.~\ref{Fig: geometric_l1}.
This increased complexity of the constraint set makes solving $\ell_1$-LMO more challenging.
From a computational point of view, $\ell_1$-LMO requires a sorting procedure \cite{garber2016linearly} to handle the complexity and hence takes up
too much time.


We also observe that LLOO/sCndG framework was largely omitted from the recent surveys \cite{bomze2021frank,braun2022conditional,pokutta2024frank} probably due to the following two reasons.
One is on the concern of computational cost per iteration discussed above.
The other is that there lacks flexibility of incorporating existing accelerating strategies such as Away-steps.
In this paper, we propose a new framework of constructing the subset $\calP_k$ that is not based on any norms.
In the meantime, the computational cost per iteration is reduced probably to minimum and there is flexibility
to include various acceleration strategies. We explain our framework below.

\begin{figure}[t]
    \centering
    \subfloat[The intersection of $\ell_1$-norm ball with the unit simplex.]{
        \label{Fig: geometric_l1}
        \includegraphics[width=0.45\linewidth]{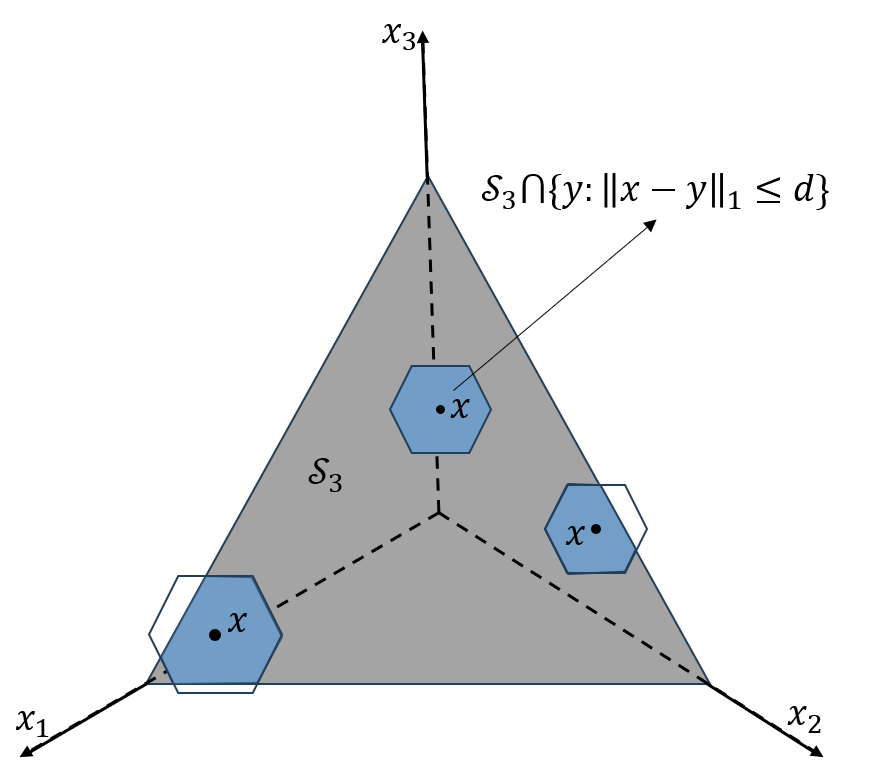}
    }
    ~
    \subfloat[The intersection of the simplex ball with the unit simplex.]{
        \label{Fig: geometric_simplex}
        \includegraphics[width=0.45\linewidth]{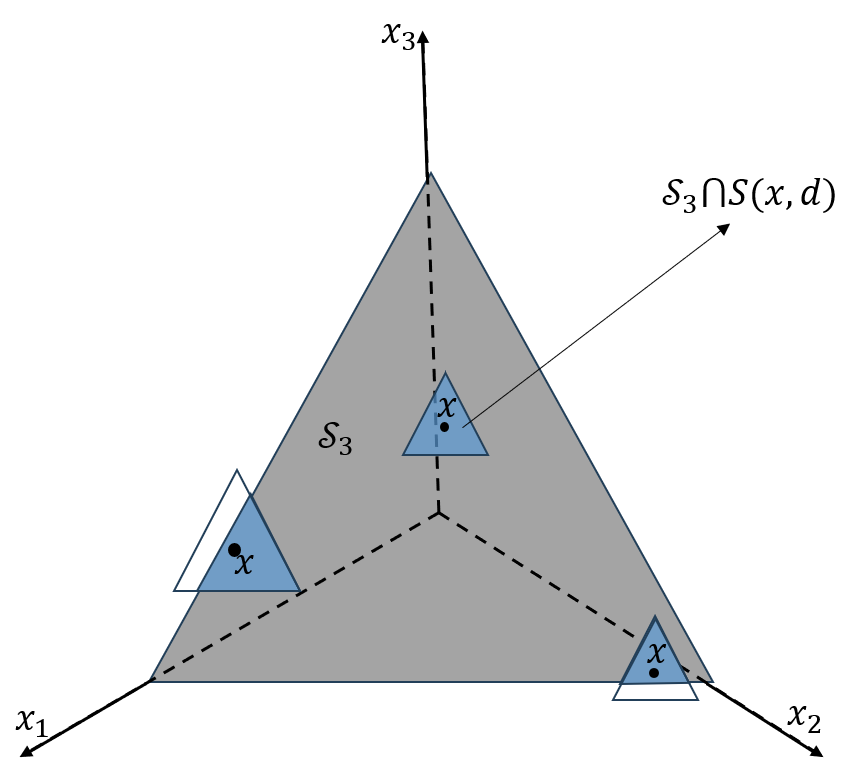}
    }
    \caption{Schematic diagram of the feasible sets for solving $\ell_1$-LMO and SLMO when $n=3$.}
    \label{Fig: geometric}
\end{figure}

\subsection{Simplex LMO and Simplex FW: a New Proposal}

Ideally, we would like our subset $\calP_k$ to be like the simplex $S_n$ so that linear optimization
over it can be fast executed.
We here introduce the {\em simplex ball} $S(\bfx, d)$ with centroid $\bfx$ and radius $d>0$ (detailed definition later).
It coincides with the \textit{atom norm} of the unit simplex as introduced by \cite{Chandrasekaran2012}. Moreover, the unit simplex $S_n$ is a simplex ball.
A very useful property is that the intersection of two simplex balls is again a simplex ball:
\[
  S(\bfx_1, d_1) \cap S(\bfx_2, d_2) = S(\bfx_3, d_3),
\]
where $\bfx_3$ and $d_3$ can be cheaply computed from $(\bfx_i, d_i)$, $i=1,2$.
This property is illustrated in  Fig.~\ref{Fig: geometric_simplex}.
Consequently, given $\bfx \in S_n$, a radius $d>0$ and $\bfc \in \mathbb{R}^n$, we define the Simplex Linear Minimization Oracle $\SLMO(\bfx, d, \bfc)$ by
\[
  \SLMO(\bfx, d, \bfc) = \argmin_\bfy \Big\{
   \langle \bfc, \bfy \rangle \ \vert \ \bfy \in S_n \cap S(\bfx, d)
  \Big\} .
\]
Since the constraint is again a simplex ball, $\SLMO$ has a closed-form formula
(see Alg.~\ref{Alg: SLMO}) and its
complexity is roughly same as $\LMO(\bfc, S_n)$.
Furthermore, we prove that $\SLMO$ is LLOO in Lemma \ref{Lemma-SLMO-LLOO}.
The consequence is that linearly convergent algorithm can be developed by following the template in
\cite{garber2016linearly}.
This part forms the first contribution of the paper.

Casting SLMO as an instance of LLOO does not benefit too much in terms of computational efficiency
because, as a common practice in FW methods, direction-correction step in \eqref{Modified-FW}
is essential in improving numerical performance.
To accommodate this need, we make two additions. 
The first one is on a flexible rule to update the radius of the simplex ball.
Any lower-bound for the objective function $f$ is permitted and the lower-bound by SLMO 
is a choice. We opt for the use of the best lower-bound available at the current iterate.   
This leads to the Simplex Frank-Wolfe (SFW) method  in Alg.~\ref{Alg: SFW},
which is proved to be linearly convergent in Thm.~\ref{Thm: Convergence-SL1}.


This second addition makes use of an important observation that SLMO can be split into two parts.
The first part is the construction of the simplex ball and the second part is linear optimization over the
simplex ball. Linear optimization is much cheaper than construction of the simplex ball.
It would be much economical if we perform linear optimization a few more times for each simplex ball:
\[
  \bfp_k \approx \argmin_{\bfy} \Big\{
    f(\bfy) \ | \ \bfy \in S_n \cap S(\bfx_{k-1}, d_{k-1})
  \Big\} .
\]
This $\bfp_k$ functions like the direction-correction used in the general framework \eqref{Modified-FW}.
This allows us to take advantage of existing FW algorithms. For example,
AFW and PFW can be used for this part.
This leads to our refined SFW method (rSFW) (see Alg.~\ref{Alg: rSFW} and Alg.~\ref{Alg: rSFW-P}).
We emphasize that the oracle used in our rSFW method requires only one additional vector addition compared to the standard LMO, whose computational complexity is probably the cheapest among all existing FW-type algorithms.
Our numerical experiments show that rSFW with Away step and Pairwise steps improves its performance
significantly.
The resulting algorithmic scheme is hence different from LLOO scheme and we provide complete convergence analysis.
This part may be treated as our second contribution.

Our third contribution is on extending the simplex case to general polytope case.
We will make use of some
fundamental connections between them
established in \cite{garber2016linearly}.
Since the simplex ball is not defined from any norm, some part of the extension
is highly non-trivial. In particular, the iteration complexity of the extended SFW depends only on the 
problem dimension $n$ instead of the number of extreme points $N$ of $\calP$, see Thm.~\ref{Thm: Convergence_sFW_P}.
Computationally, this can all be achieved for three popular polytopes: Hypercube, Flow polytope and $\ell_1$-ball.

Our final part is to address the implementation issues including adaptive backtracking techniques on choosing the
problem prameters $L$ and $\mu$, and incorporating Away-FW and Pairwise-FW steps to SFW methods. 
Numerical experiments demonstrate that our SFW methods are highly competitive.

\subsection{Organization}

In the preceding discussion, we only focus on the framework \eqref{Modified-FW} that may lead to
linearly convergent FW methods. We avoid specifying the actual conditions on $f$ because various conditions can ensure such linear rate.
In Section~\ref{Section:Background}, we describe such conditions as well as some background on polytopes. We will explain the key concept of LLOO proposed in \cite{garber2016linearly}.
Section~\ref{Section-SFW} contains the detailed development of SLMO and the resulting Simplex FW methods
(SFW and rSFW) for the case $\calP= S_n$.
The extension to the general polytope case is conducted in Section~\ref{Section: gen_P}.
Lengthy proofs are moved to the Supplement for the benefit of readability of the paper.
Section~\ref{Section-Numerical} reports some illustrative examples to demonstrate the
advantage of SFW methods over some existing algorithms.
We conclude the paper in Section~\ref{Section-Conclusion}.

\section{Notation and Background} \label{Section:Background}

\subsection{Notation}
We employ lower-case letters, bold lower-case letters, and capital letters to denote scalars, vectors, and matrices, respectively (e.g., $x,\bfx$ and $X$).
For two column vectors $\bfx, \bfy \in \mathbb{R}^n$, $\max\{\bfx, \bfy\}$ is a new
column vector that takes the component-wise maximum of $\bfx$ and $\bfy$.
The vector $\min\{\bfx, \bfy\}$ is similarly defined.
For vectors, we denote the standard Euclidean norm by $\|\cdot\|$ and the standard inner-product by $\langle\cdot,\cdot\rangle$.
For  a vector $\bfx \in \mathbb{R}^n$, a subset $C \subseteq \mathbb{R}^n$, and $\tau >0$, we define
\[
  \bfx + C := \left\{
   \bfx + \bfy \ | \ \bfy \in C
  \right\} \qquad \mbox{and} \qquad
  \tau C := \left\{
   \tau \bfy  \ | \ \bfy \in C
  \right\},
\]
where ``$:=$'' means ``define''.

We let $\mathbb{B}(\bfx, r)$ to denote the Euclidean ball of radius $r$ centered at $\bfx$.
For matrices, we let $\|\cdot\|$ denote the spectral norm.
For a vector $\bfx\in\mathbb{R}^n$, we use $x_i$ or $x(i)$ to denote the $i$-th component.
For a matrix $A$, we use $A(i)$ to denote the $i$-th row of $A$.
The vector $\bfone_n$ represents a vector with all entries equal to 1, and
 $\bfe_i$ is the standard $i$th unit vector in $\mathbb{R}^n$ which takes value $1$ at its $i$th position and
 $0$ elsewhere.
Given a set $\mathcal{V}$, we denotes its convex hull as $\Conv\{\mathcal{V}\}$.
For any positive integer $n$, we use the notation $[n]$ to represent the set $\{1,\dots,n\}$.
We use $S_n:=\{\bfx\in\mathbb{R}^n\vert \sum_{i=1}^nx_i=1, \bfx\geq 0 \}$ to denote the unit simplex.

\subsection{Smoothness, Strong Convexity and Stepsizes}

Throughout the paper, we will assume $L$-smoothness and $\mu$-strong convexity of $f$.

\begin{definition}[Smooth function]
    We say that a function $f(\bfx):\mathbb{R}^{n}\to \mathbb{R}$ is $L$-smooth over a convex set $\calP\subset \mathbb{R}^n$, if for every $\bfx,\bfy\in\calP$ there holds
    \begin{equation*}
        f(\bfy)\leq f(\bfx)+\langle \bfy-\bfx,\nabla f(\bfx)\rangle + \frac{L}{2}\|\bfx-\bfy\|^2.
    \end{equation*}
\end{definition}
\begin{definition}[Strongly convex function]
    We say that a function $f(\bfx):\mathbb{R}^{n}\to \mathbb{R}$ is $\mu$-strongly convex over a convex set $\calP\subset \mathbb{R}^n$, if for every $\bfx,\bfy\in\calP$ there holds
    \begin{equation*}
        f(\bfy)\geq f(\bfx)+\langle \bfy-\bfx,\nabla f(\bfx)\rangle + \frac{\mu}{2}\|\bfx-\bfy\|^2.
    \end{equation*}
\end{definition}
The above definition combined with first order optimality conditions imply that for a $\mu$-strongly convex function $f$, if $\bfx^*=\argmin_{\bfx\in\calP}f(\bfx)$, then for any $\bfx\in\calP$ we have
\begin{equation}\label{Eq: strongly_convex_property}
  f(\bfx)-f^* \ge \frac{\mu}2  \|\bfx-\bfx^*\|^2 .
\end{equation}
This property, while weaker than strong convexity, is essential for demonstrating linear convergence rather than relying solely on strong convexity. A natural generalization of this property is known as the quadratic growth property.

There are three popular step size strategies:
\begin{enumerate}
    \item Simple step size:
        \begin{equation}\label{Eq: step size_const}
            \delta_k = 2/(k+1),\quad k=1,\dots.
        \end{equation}
    \item Line-search step size:
        \begin{equation}\label{Eq: step size_line}
            \delta_k = \operatorname*{arg\,min}_{\delta\in[0,1]}f((1-\delta)\bfx_{k-1}+\delta\bfy_{k}),\quad k=1,\dots.
        \end{equation}
    \item Short step size:
        \begin{equation}\label{Eq: step size_smooth}
            \delta_k = \min\left\{1, \frac{\langle \nabla f(\bfx_{k-1}) ,\bfx_{k-1}-\bfy_k\rangle}{L\|\bfx_{k-1}-\bfy_k\|^2}\right\},\quad k=1,\dots.
        \end{equation}
\end{enumerate}

For the three step size strategies described above, it can be shown that the standard Frank-Wolfe method exhibits the following convergence rates. For a detailed proof, refer to the modern surveys by \cite{jaggi2013revisiting}, \cite{lan2013complexity} or \cite{freund2016new}.

\begin{theorem}\label{Thm: FW}
    Let $\{\bfx_k\}$ be the sequences generated by standard FW method with step size policy for $\{\delta_k\}$ in \eqref{Eq: step size_const}, \eqref{Eq: step size_line}, or \eqref{Eq: step size_smooth}. Then, for $k\geq 1$, we have
    \begin{equation}\label{Eq: convergence_FW}
        f(\bfx_k) -f^*\leq \langle\nabla f(\bfx_k),\bfx_k-\bfy_{k+1} \rangle\leq 2L D^2/(k+1),
    \end{equation}
    where $\bfy_{k+1}=\text{LMO}(\nabla f(\bfx_k),\calP)$ and $D$ is the diameter of $\calP$.
\end{theorem}


\subsection{Quantities of Polytope}\label{Subsection: Quantity_Polytope}

The quantities reviewed in this part are well defined and investigated in \cite{garber2016linearly}
and they are mainly used in the extension of \SFW\ to polytopes.

Let $\calP$ be a polytope described by linear equations and inequalities, specifically $\calP=\{\bfx\in\mathbb{R}^n| A_1\bfx=\bfb_1, A_2\bfx\leq \bfb_2\}$, where $A_2\in\mathbb{R}^{m\times n}$.
Without loss of generality, we assume that all rows of $A_2$ have been scaled to possess a unit $l_2$ norm.
We denote the set of vertices of $\calP$ as $\mathcal{V}(\calP)$ and let $N=|\mathcal{V}(\calP)|$ represent the number of vertices.

Next, we introduce several geometric parameters related to $\calP$ that will naturally arise in our analysis.
The Euclidean diameter of $\calP$ is defined as $D(\calP)=\max_{\bfx,\bfy\in\calP}\lVert \bfx-\bfy\rVert$.
We define
\begin{equation*}
    \xi(\calP)=\min_{\bfv \in \mathcal{V}(\calP)}\left(\min \left\{{b}_2(j)-A_2(j) \bfv \mid j \in[m], A_2(j) \bfv<b_2(j)\right\}\right).
\end{equation*}
This means that for any inequality constraint defining the polytope and for a given vertex, that vertex either satisfies the constraint with equality or is at least $\xi(\calP)$ units away from satisfying it with equality.
Let $r(A_2)$ denote the row rank of $A_2$, and let $\mathbb{A}(\calP)$ represent the set of all $r(A_2)\times n$ matrices with linearly independent rows selected from the rows of $A_2$.
We then define $\psi(\calP)=\max _{M \in \mathbb{A}(\calP)}\|M\|$. Finally, we introduce condition number of $\calP$ as
\begin{equation}\label{eq: eta_def}
    \eta(\calP)=\psi(\calP)D(\calP)/\xi(\calP).
\end{equation}
It is important to note that the translation, rotation and scaling of the polytope $\calP$ are invariant to $\eta(\calP)$.
For convenience we use $\mathcal{V},D,\xi,\psi,\eta$ without explicitly mentioning the polytope when $\calP$ is clear from context.
It is worth noting that in many relevant scenarios---particularly in cases where efficient algorithms exist for linear optimization over the given polytope---estimating the parameters $D,\xi,\psi$ is often straightforward.
This is particularly true in convex domains encountered in combinatorial optimization, such as flow polytopes, matching polytopes, and matroid polytopes, among others.
Furthermore, our algorithm relies primarily on the parameter $\eta$ and $D$.

\subsection{LLOO}

A major concept proposed by Garber and Hazan \cite[Def.~2.5]{garber2016linearly} is LLOO.
Consider the problem \eqref{Problem: general}. We say a procedure $\calA(\bfx, d, \bfc)$, where
$\bfx \in \calP$, $d>0$, $\bfc \in \mathbb{R}^n$, is an LLOO with parameter $\rho\ge 1$ for polytope
$\calP$ if $\calA(\bfx, d, \bfc)$ returns a feasible point $\bfp \in \calP$ such that
\begin{itemize}
	\item[(i)] $\langle \bfy, \bfc \rangle \ge \langle \bfp, \bfc \rangle$
	for all $\bfy \in \mathbb{B}(\bfx, d) \cap \calP$, and
	\item[(ii)] $\| \bfx - \bfp \| \le \rho d$.
\end{itemize}

Suppose the optimal solution $\bfx^*$ of \eqref{Problem: general} is contained in
$\mathbb{B}(\bfx, d)$ and LLOO $\calA(\bfx, d, \nabla f(\bfx))$ return a
feasible point $\bfp \in \calP$. The convexity of $f$ implies the following.
\begin{align*}
	f(\bfx^*) & \ge f(\bfx) + \langle \nabla f(\bfx), \bfx^* - \bfx \rangle \\
	&\ge f(\bfx)  + \langle \nabla f(\bfx), \bfp - \bfx \rangle  \quad \mbox{(by $\bfx^* \in \mathbb{B}(\bfx, d)$ and Property LLOO(i))}
\end{align*}
That is, LLOO naturally provides a lower bound for the optimal objective.
Such lower bounds will be used in our updating scheme of the radius $d$.
We also note that LLOO $\calA(\bfx, d, \nabla f(\bfx))$ often return an optimal
solution over a subset $\calP_k$, which should be constructed in \eqref{Modified-FW} bearing in
mind of its solution efficiency.

Given an LLOO procedure available, a general FW framework can be developed for it to enjoy
a linear convergence rate over general polytope $\calP$ provided $f$ being $L$-smooth and $\mu$-strongly convex, see \cite[Thm.~4.1]{garber2016linearly}.
It is proved that $\ell_1$-LMO is an LLOO over the simplex polytope $S_n$.
The framework is then extended to general polytope.
As we discussed in Introduction, $\ell_1$-LMO is much less efficient than
the original LMO over the simplex polytope $S_n$.
This is the one of the motivations for us to develop the simplex LMO below.

\section{Simplex FW Method}\label{Section-SFW}

This section is solely devoted to the case of simplex polytope: Problem \eqref{Problem: general} with 
$\calP  = S_n$.
We will then extend the obtained results to general polytopes in the next section.
We start with the introduction of simplex ball.

\subsection{Simplex Ball and Simplex-based Linear Minimization Oracle}\label{Section: Simplex ball}

In this subsection, we will formally define the concept of the simplex ball and present some of its useful properties. Following this, we will introduce the Simplex-based Linear Minimization Oracle (SLMO) and provide an efficient algorithm for solving it.


\begin{definition}[Simplex ball]\label{Def: simplex_ball}
	Let
	$
	  S_0 := S_n - \frac 1n \bfone_n .
	$    
	For any $\bfx\in\mathbb{R}^n$ and $d>0$, we define $S(\bfx, d)$ as the simplex ball of radius $d$ centered at $\bfx$ by
    \begin{equation}\label{SimplexBall-New}
    	S(\bfx, d) := \bfx + (nd) S_0 = \Big\{ (\bfx - d\bfone_n) + nd\bflambda \mid \bflambda \in S_n  \Big\}.
    \end{equation}
\end{definition}

The following properties of the simplex ball are crucial to our development.
The proof is moved to the Supplement \ref{Section: app_proof_simplexball}.

\begin{lemma}\label{Lemma: Simplex ball}
    Given $\bfx\in S_n$ and $d>0$, we have
    \begin{enumerate}
        \item[(1)] The unit simplex is a simplex ball, i.e., $S_n = S(\frac{1}{n}\bfone_n, \frac{1}{n})$.
        Moreover, we have
        \begin{equation}\label{SimplexBall}
        	 S(\bfx,d) =\Big\{\bfx+d\bfr\vert \bfr\in \rm{Conv}\{n\bfe_i-\bfone_n:i\in [n] \} \Big\}.
        \end{equation}

        \item[(2)]
        The intersection of two simplex balls, if nonempty, is again a simplex ball. In particular,
        \be \label{Eq: definition_d}
         S_n\cap S(\bfx,d)=S(\widehat{\bfx},\widehat{d}) \ \
         \mbox{where} \ \
         \left\{
         \begin{array}{l}
         	\widehat{d} = \frac{\sum_{i=1}^n\min\{d,\; x_i\}}{n} \\
         	\widehat{x}_i = \max\{x_i, d\}+\widehat{d}-d, \quad i \in [n] .
         \end{array}
         \right .
        \ee

        Moreover, for $\bfx_1,\bfx_2\in S_n$ and radius $d_1,d_2>0$ such that
         $S(\bfx_1,d_1)\cap S(\bfx_2,d_2)\neq\emptyset$, it holds
         \[
         S(\bfx_1, d_1)\cap S(\bfx_2,d_2)=S(\bfx_3, d_3),
         \]
         where
         \begin{equation}\label{Eq: intsect_simplex_balls}
         	\begin{aligned}
         		d_3   &= \frac{1+\sum_{i=1}^n\min\{d_1-x_1(i),d_2-x_2(i) \}}{n},\\
         		x_3(i)  &= \max\{x_1(i)-d_1,x_2(i)-d_2 \}+ d_3, \ \ i\in[n]
         	\end{aligned}
         \end{equation}
         Consequently, we have $d_3 \le \min\{d_1, d_2\}$.

        \item[(3)] 
        The linear optimization over a simplex ball has the following closed-form solution:
        \begin{align*}
          & \bfy^* := \bfx+ (nd) \Big( \bfe_{i^*}- \frac{\bfone_n}n \Big) \in  \argmin_{\bfy\in S(\bfx,d)}\ \langle \bfc, \bfy\rangle \ \ \mbox{with} \ \ i^*=\argmin_{i\in [n]} c_i.
        \end{align*}


        \item[(4)] The diameter of the simplex ball $S(\bfx,d)$ is $\sqrt{2}nd$, $\mbox{i.e.}$, $\max_{\bfy_1,\bfy_2\in S(\bfx,d)}\lVert \bfy_1-\bfy_2 \rVert = \sqrt{2}nd.$

        \item[(5)] For any point $\bfy\in S_n$, if $\lVert \bfx-\bfy\rVert\leq d$, then $\bfy\in S(\bfx,d)$. Moreover, for any point $\bfy\in S(\bfx,d)$, we have $\lVert \bfy-\bfx\rVert \leq nd$.
    \end{enumerate}
\end{lemma}

We now give a formal definition of our LMO based on simplex ball.

\begin{definition}[SLMO]\label{Def: SLMO}
    Given a linear objective $\bfc\in\mathbb{R}^n$, radius $d>0$ and a point $\bfx\in S_n$, a solution $\bfy^*\in \rm{SLMO}(\bfx,d,\bfc)$ is called simplex-based linear minimization oracle if it solves the following optimization problem
    \begin{equation}\label{Problem: SLMO}
    \begin{aligned}
        &\min\ \langle \bfy,\bfc\rangle \quad
        &\rm{s.t.} \quad  \bfy\in S(\bfx,d)\cap S_n.
    \end{aligned}
    \end{equation}
\end{definition}

We have proved in Lemma~\ref{Lemma: Simplex ball}(5) that $\mathbb{B}(\bfx, d) \subseteq S(\bfx, d) \subseteq \mathbb{B}(\bfx, nd)$.
Therefore, for any $\bfy \in \mathbb{B}(\bfx, d) \cap S_n$, we must have
$
 \langle \bfy,\bfc\rangle \ge \langle \bfy^*,\bfc\rangle .
$
This is the first property of LLOO. Moreover,
since both $\bfx, \bfy^* \in S(\bfx, d)$, we must have $\| \bfx - \bfy^*\| \le \rho d$ with $\rho = n$.
This leads to the following key result.

\begin{lemma} \label{Lemma-SLMO-LLOO}
Given $\bfx \in \calP$, $d>0$ and $\bfc \in \mathbb{R}^n$
such that $S(\bfx,d)\cap {S}_n \not= \emptyset$, then
$\rm{SLMO}(\bfx, d, \bfc)$ is an \rm{LLOO} $\calA(\bfx, d, \bfc)$ with $\rho=n$.
\end{lemma}

The implication of this result is far-reaching because the framework developed in
\cite{garber2016linearly} can be followed to get a linearly convergent algorithm
with SLMO.
An even more important result is that SLMO problem
\eqref{Problem: SLMO} can be solved by the following simple algorithm.

\begin{algorithm}[H]
\footnotesize
	\renewcommand{\algorithmicrequire}{\textbf{Input:}}
	\renewcommand{\algorithmicensure}{\textbf{Output:}}
	\caption{$\SLMO(\bfx, d, \bfc)$}
	\label{Alg: SLMO}
	\begin{algorithmic}[1]
        \REQUIRE point $\bfx\in S_n$, linear objective $\bfc\in\mathbb{R}^n$, radius $d>0$.
        \STATE $\widehat{d}\gets\frac{\sum_{i=1}^n\min\{d, x_i\}}{n}$
        \STATE $\widehat{\bfx}\gets \bfx-\min\{\bfx,d\bfone_n\}+\widehat{d}\bfone_n$
        \STATE $\bfy_+\gets \widehat{\bfx}-\widehat{d}\bfone_n$
        \STATE $i^* \gets \argmin_{i\in [n]}c_i$
        \STATE $\bfy^*\gets\bfy_++n\widehat{d}\;\bfe_{i^*}$
        \ENSURE $\bfy^*$.
	\end{algorithmic}
\end{algorithm}

The algorithm follows these basic steps.
Firstly, it represents the constraint set as a single simplex ball:
$
 S(\widehat{\bfx}, \widehat{d}).
$
Secondly, it uses the existing theoretical results of linear programming over the simplex ball to find the optimal solution.

\begin{lemma} \label{Lemma-SLMO}
    Alg.~\ref{Alg: SLMO} finds an optimal solution to Problem \eqref{Problem: SLMO}.
\end{lemma}
\begin{proof}
    First, by Lemma \ref{Lemma: Simplex ball}(2), we have \(S(\bfx, d) \cap S_n = S(\widehat{\bfx}, \widehat{d})\), where the definitions of \(\widehat{\bfx}\) and \(\widehat{d}\) are given in \eqref{Eq: definition_d}. Consequently, Problem \eqref{Problem: SLMO} is equivalent to \(\min_{\bfy \in S(\widehat{\bfx}, \widehat{d})} \langle \bfy, \bfc \rangle\). Note that this is the same form as the problem in Lemma \ref{Lemma: Simplex ball}(3). Thus, we have
    \[
    \bfy^* = \widehat{\bfx} + \widehat{d}(n\bfe_{i^*} - \bfone_n) =
    \max\{ \bfx, d\bfone_n\}
   - d\bfone_n + n\widehat{d}\;\bfe_{i^*}.
    \]
    Consequently, Alg.~\ref{Alg: SLMO} solves Problem \eqref{Problem: SLMO}.
\end{proof}

\begin{remark} (Comparison with $\LMO(\bfc, S_n)$ and $\ell_1$-$\LMO(\bfx, d, \bfc)$)
If we treat the element-wise minimum between two vectors as a basic operation, then $\SLMO$ requires only one extra basic operation, one more vector summation, and one more vector addition compared to the the original $\LMO$ over the simplex $S_n$.
Therefore, its total exact complexity is $4n$ {\em flops}, making it nearly as efficient as $\LMO(\bfc, S_n)$.
However, $\ell_1$-$\LMO(\bfx, d, \bfc)$ involves a sorting operation, whose
overall complexity is usually $O(n\log(n)$), It also involves a few more vector additions.
As will be illustrated in Fig.~\ref{Fig: LMO_compare_L1}, it is far less efficient than the original $\LMO(\bfc, S_n)$ and \SLMO.
Therefore, we expect that a linearly convergent algorithm with $\SLMO$ should be
efficient as well. We develop it below.
\end{remark}

\subsection{SFW: Simplex Frank-Wolfe Method}\label{Section: sFW}

In this subsection, we present a new variant of Frank-Wolfe method called Simplex Frank-Wolfe (abbreviated as SFW), obtained by replacing the LMO with SLMO. The algorithm is formally described as follows.

\begin{algorithm}[H]
\footnotesize
\renewcommand{\algorithmicrequire}{\textbf{Input:}}
\renewcommand{\algorithmicensure}{\textbf{Output:}}
\caption{Simplex Frank-Wolfe Method: SFW}
\label{Alg: SFW}
\begin{algorithmic}[1]
    \REQUIRE $\bfx_0\in S_n$, initial lower bound $B_0$.
    \STATE Set: $d_0\gets\sqrt{\frac{2(f(\bfx_0)-B_0)}{\mu}}.$
    \FOR{$k=1,\dots$}

        \STATE Compute $\bfy_k\in \SLMO({\bfx}_{k-1},{d}_{k-1},\nabla f(\bfx_{k-1}))$.

        \STATE Compute the working lower bound: $B_k^w\gets f(\bfx_{k-1})+\langle\nabla f(\bfx_{k-1}),\bfy_k-\bfx_{k-1}\rangle$.

        \STATE Update best bound $B_k\gets \max\{B_{k-1}, B_k^w \}$.

        \STATE Set $\bfx_k\gets (1-\delta_k)\bfx_{k-1}+\delta_k\bfy_k$ for some $\delta_k\in [0,1]$.


        \STATE Set: $d_k\gets \sqrt{\frac{2(f(\bfx_k)-B_k)}{\mu}}$.
    \ENDFOR
\end{algorithmic}
\end{algorithm}

Before stating its convergence rate result, we make the following remarks regarding Alg.~\ref{Alg: SFW}.

\begin{remark} \label{Remark-Alg-SFW}
(Choice of $d_k$ update strategy)
We could follow the linear shrinking rule of $d_k$ in \cite{garber2016linearly,lan2020first}:
$d_k = \gamma d_{k-1}$ with properly chosen $\gamma <1$. The linear convergent rate would be guaranteed by invoking the LLOO property of \SLMO\; (Lemma~\ref{Lemma-SLMO-LLOO}).
We do not take this route for convergence analysis because of the following two reasons.
One is that the key property $\mathbb{B}(\bfx, d) \subseteq S(\bfx, d) \subseteq \mathbb{B}(\bfx, nd)$ ensuring LLOO will have to be modified when it comes to the general polytope as our simplex ball is not based on any norm.
The corresponding relationship becomes $\mathbb{B}(\bfx, dD/\eta) \subseteq S_{\calP}(\bfx, d) \subseteq \mathbb{B}(\bfx, (n+1)dD)$, see Lemma~\ref{Lemma: SLMO_P}, where $S_\calP$ is defined.
At least at a technical level, the original LLOO will have to be generalized to suit this extension and the corresponding proofs have also to be reproduced. The proof we provided below is more direct.
The other reason is that we use the best lower bound $B_k$ provided by the algorithm to define $d_k$.
This choice is important because we are going to incorporate other accelerating strategies to \SFW\ resulting in \rSFW.
The stopping criterion used there will be also based on the best lower bounds obtained.	
Therefore, the convergence analysis for \SFW\ will naturally be adapted to \rSFW.
\end{remark}

At the $k$-th iteration, Alg.~\ref{Alg: SFW} first invokes SLMO to find the minimum point $\bfy_k$ of the first-order approximate expansion of the objective function within the region $S(\bfx_{k-1},d_{k-1})\cap S_n$.
Subsequently, using a suitable step size $\delta_k$, a convex combination of $\bfy_k$ and $\bfx_{k-1}$ is computed to update the iteration point to $\bfx_{k+1}$.
Finally, the algorithm updates the radius $d$, ensuring that the optimal solution $\bfx^*$ progressively falls within a smaller neighborhood $S(\bfx_{k},d_{k})$.
For Alg.~\ref{Alg: SFW}, we propose the following simple step size, as an alternative to the simple step size selection in the original Frank-Wolfe algorithm:
\begin{equation}\label{Eq: step size_constant2}
    \delta_k=\frac{\mu}{2Ln^2}.
\end{equation}
We have the following linear convergence rate result. The induction technique in the proof below was taken from \cite[Lemma~4.3]{garber2016linearly}.

\begin{theorem}\label{Thm: Convergence-SL1}
    Let $\{\bfx_k\}$ be the sequences generated by Alg.~\ref{Alg: SFW} with step size policy for $\{\delta_k\}$ in \eqref{Eq: step size_line}, \eqref{Eq: step size_smooth}, or \eqref{Eq: step size_constant2}. Then, for $k\geq 0$, we have
    \begin{equation}\label{Eq: Convergence_sfw}
        f(\bfx_k) -f^*\leq f(\bfx_k) - B_k\leq \frac{\mu d_0^2}{2}e^{-\frac{\mu}{4Ln^2}k}.
    \end{equation}
\end{theorem}
\begin{proof}
    We first claim that $\bfx^*\in S(\bfx_k,d_k)$ and that $f(\bfx_k)-B_k\leq \frac{\mu d_k^2}{2}$. We prove this by induction.
    First, we have
    \begin{equation*}
        \frac{\mu d_0^2}{2}=f(\bfx_0)-B_0\geq f(\bfx_0)-f^*\stackrel{(a)}{\geq} \frac{\mu}{2}\lVert \bfx_{0}-\bfx^* \rVert^2,
    \end{equation*}
    where $(a)$ comes from \eqref{Eq: strongly_convex_property}.
    This implies that $\lVert \bfx_{0}-\bfx^* \rVert\leq d_0$, and by Lemma \ref{Lemma: Simplex ball}(5), we have $\bfx^*\in S(\bfx_0,d_0)$.
    Therefore, the claim holds for $k=0$.

    Now suppose that $\bfx^*\in S(\bfx_t,d_t)$ and $f(\bfx_t)-B_t\leq \frac{\mu d_t^2}{2}$ for all $t\leq k-1$.
    Let $\gamma :=\frac{\mu}{2Ln^2} \le 1$. For step size policy $\delta_k$ in \eqref{Eq: step size_line}
    (exact line search stepsize)
     or \eqref{Eq: step size_constant2}, we both have
    \begin{eqnarray}\label{Eq: detail2_1}
        f(\bfx_k)
        &= & f(\bfx_{k-1}+\delta_k (\bfy_k-\bfx_{k-1}))   \le f(\bfx_{k-1}+\gamma (\bfy_k-\bfx_{k-1})) \nonumber \\
        &\leq & f(\bfx_{k-1})+\gamma \langle \nabla f(\bfx_{k-1}),\bfy_k-\bfx_{k-1}\rangle+\frac{L\gamma^2}{2}\lVert \bfy_k-\bfx_{k-1} \rVert^2.
    \end{eqnarray}
    Similarly, for the step size policy \eqref{Eq: step size_smooth} (short stepsize), we have
    \begin{eqnarray}\label{Eq: detail2_2}
        f(\bfx_k)
        &\leq & f(\bfx_{k-1})+\delta_k\langle \nabla f(\bfx_{k-1}),\bfy_k-\bfx_{k-1}\rangle+\frac{L\delta_k^2}{2}\lVert \bfy_k-\bfx_{k-1} \rVert^2 \nonumber \\
        &\leq & f(\bfx_{k-1})+\gamma \langle \nabla f(\bfx_{k-1}),\bfy_k-\bfx_{k-1}\rangle+\frac{L\gamma^2}{2}\lVert \bfy_k-\bfx_{k-1} \rVert^2.
    \end{eqnarray}
    Combining \eqref{Eq: detail2_1} and \eqref{Eq: detail2_2}, we have
    \begin{equation*}
    \begin{aligned}
        f(\bfx_k)
        \leq & f(\bfx_{k-1})+\gamma \langle \nabla f(\bfx_{k-1}),\bfy_k-\bfx_{k-1}\rangle+\frac{L\gamma^2}{2}\lVert \bfy_k-\bfx_{k-1} \rVert^2 \\
        \stackrel{(b)}{\leq} & (1-\gamma)f(\bfx_{k-1})+\gamma B_k^w+\frac{L\gamma^2}{2}\lVert \bfy_k-\bfx_{k-1} \rVert^2 \\
        \stackrel{(c)}{\leq}\ & (1-\gamma)(f(\bfx_{k-1})-B_{k-1})+B_k+\frac{L\gamma^2}{2}\lVert \bfy_k-\bfx_{k-1} \rVert^2
    \end{aligned}
    \end{equation*}
    holds for step size policy \eqref{Eq: step size_line}, \eqref{Eq: step size_smooth}, or \eqref{Eq: step size_constant2}, where $(b)$ comes from the definition of $B_k^w$, and $(c)$ is due to $B_k\geq \max\{B_{k-1},B_k^w\}$.
    Subtracting $B_k$ from the both sides of the above inequality, we obtain
    \begin{equation*}
    \begin{aligned}
        f(\bfx_k)-B_k\leq & (1-\gamma)(f(\bfx_{k-1})-B_{k-1})+\frac{L\gamma^2}{2}\lVert \bfy_k-\bfx_{k-1} \rVert^2 \\
        \stackrel{(d)}{\leq} & (1-\gamma)\frac{\mu}{2}d_{k-1}^2+\frac{L\gamma^2}{2}n^2d_{k-1}^2 = \left[ (1-\gamma)\frac{\mu}{2} + \frac{L\gamma^2n^2}{2} \right]d_{k-1}^2,
    \end{aligned}
    \end{equation*}
    where $(d)$ is due to our inductive hypothesis and Lemma \ref{Lemma: Simplex ball}(5).
    By plugging in the value of $\gamma$, and using $1-x\leq e^{-x}$, we have that
    \begin{equation} \label{Gap-Bound}
        f(\bfx_k)-B_k\leq \frac{\mu}{2} \left(1-\frac{\mu}{4Ln^2} \right)d_{k-1}^2\leq \frac{\mu}{2}e^{-\frac{\mu}{4Ln^2}}d_{k-1}^2.
    \end{equation}
    With the definition of $d_k$, we have $f(\bfx_k)- B_k = \frac{\mu}2 d_k^2$. The bound in \eqref{Gap-Bound}
    implies
    \be \label{d-LinearReduction}
      d_k \le e^{-\frac{\mu}{8Ln^2}}{d}_{k-1} .
    \ee
    By the inductive hypothesis, we know that $\bfx^*\in S(\bfx_t,d_t)$ holds for all $t\leq k-1$. Thus $B_{t+1}^w$ is a valid lower bound of $f^*$, and consequently, $B_k$ is also a lower bound of $f^*$. Now by  \eqref{Eq: strongly_convex_property}, we have
    \begin{equation*}
        \lVert \bfx_k-\bfx^*\rVert^2\leq \frac{2}{\mu}(f(\bfx_k)-f^*)\leq \frac{2}{\mu}(f(\bfx_k)-B_k)= d_k^2.
    \end{equation*}
    This implies that $\bfx^*\in S(\bfx_k,d_k)$ by Lemma \ref{Lemma: Simplex ball}(5). Therefore, we have completed the proof of the claim.

    We now start to prove the conclusion in Theorem \ref{Thm: Convergence-SL1}.
    From the earlier proof, we know that $B_k\leq f^*$, thus confirming the first part of the inequality.
    By the definition of $d_k$ and the established reduction inequality \eqref{d-LinearReduction},
    we have
    $$
    f(\bfx_k)-B_k = \frac{\mu d_k^2}{2} \leq \frac{\mu d_0^2}{2}e^{-\frac{\mu}{4Ln^2}k}.
    $$
    The proof is thus completed.
\end{proof}

\begin{remark} \label{Remark-ComputationComplexity}
(Iteration complexity of SFW)
If we skip Lines 4–5 and replace Line 7 in Alg.~\ref{Alg: SFW} with $\sqrt{\frac{2\langle\nabla f(\bfx_{k-1}),\bfx_{k-1}-\bfy_k\rangle}{\mu}}$, the algorithm remains correct and preserves its convergence guarantee.
This modification avoids computing the objective value $f(\bfx_k)$, making the algorithm more practical and simple when the objective is expensive or difficult to evaluate. 
In this case, assuming that there is an oracle to obtain the gradient information $\nabla f(\bfx)$ at each iteration, we are able to give the exact number of {\em flops} operations to compute the next iterate.
The SLMO part to get $\bfy_k$ is $4n$ {\em flops}, and computing the direction $d_k$ requires an additional $3n$ {\em flops}.
For the short step size strategy, evaluating the step size $\delta_k$ incurs $2n$ {\em flops}, and updating the iterate $\bfx_k$ takes another $3n$ {\em flops}.
Thus, SFW needs $10n$ flops with a simple step size, or $12n$ flops with the short step size—--yet still achieves linear convergence for the simplex.
To our knowledge, this is the lowest per-iteration cost among FW variants with linear convergence.
\end{remark}

\subsection{Refining SFW}

In this subsection, we aim to further reduce the computational overhead of the proposed oracle as much as possible, while retaining the linear convergence rate.
The motivation stems from an important observation.
$\SLMO(\bfx, d, \bfc)$ can be split into two parts.
The first part is to construct a new Simplex-ball $S(\widehat{\bfx}, \widehat{d})=
S(\bfx, d)\cap S_n$.
For easy reference, we call it $\SLMO$-1, which corresponds to Lines 1 in Alg.~\ref{Alg: SLMO}.
The second part, which finds an optimal solution over $S(\widehat{\bfx}, \widehat{d})$, is referred to as $\SLMO$-2 and corresponds to Lines 4-5 in Alg.~\ref{Alg: SLMO}.
It is easy to see that the computation of $\SLMO$-2 requires only one more vector addition compared to the standard LMO over $S_n$ and hence its computation is already kept minimum.
The extra computation of $\SLMO$ is from $\SLMO$-1.
Since the new Simplex ball is already constructed, we like to carry out a few more times of the $\SLMO$-2 part.
This is roughly to find an approximate solution $\bfp$ to the problem
\[
  \min \; f(\bfy) \quad \mbox{s.t.} \quad \bfy \in S(\widehat{\bfx}, \widehat{d})
\]
with starting point $\bfx$. Our control of this refinement step is for
the overall computation to remain $O(n)$ and the overall convergence rate to remain linear.
The overall algorithm is given in Alg.~\ref{Alg: rSFW} and it is called Refined SFW.


\begin{algorithm}[t]
\footnotesize
\renewcommand{\algorithmicrequire}{\textbf{Input:}}
\renewcommand{\algorithmicensure}{\textbf{Output:}}
\caption{rSFW: Refined Simplex Frank-Wolfe Method}
\label{Alg: rSFW}
\begin{algorithmic}[1]
    \REQUIRE Radius contraction ratio $\rho>1$, initial lower bound $B_0$.
    \STATE Set: $d_0\gets\frac{1}{n}, \bfx_0\gets\frac{\bfone_n}{n},J\gets\frac{8\rho^2n^2L}{\mu},\bar{\bfx}_0\gets\bfx_0,\bar{d}_0\gets d_0$.
    \FOR{$k=1,\dots$}
        \STATE Set: $\bfp_0\gets\bfx_{k-1},C_0\gets B_{k-1}$.
        \STATE ($\SLMO$-1) construct the new Simplex ball: $S(\widehat{\bfx}_{k-1},\widehat{d}_{k-1})=S({\bar{\bfx}}_{k-1},\bar{d}_{k-1})\cap S_n$.
        \FOR{$j=1,\dots,J$}
            \STATE ($\SLMO$-2):  Compute $\bfy_j={\argmin}_{\bfy\in S(\widehat{\bfx}_{k-1},\widehat{d}_{k-1})}\langle \nabla f(\bfp_{j-1}),\bfy \rangle$.
            \STATE Compute the current lower bound: $C_j^w\gets f(\bfp_{j-1})+\langle \nabla f(\bfp_{j-1}),\bfy_j-\bfp_{j-1} \rangle$.
            \STATE Update the best lower bound
            $C_j\gets \max\{C_{j-1}, C_j^w\}$.
            \IF{$f(\bfp_j)-C_j\leq \frac{\mu}{2\rho^2}\widehat{d}_{k-1}^2$}
                \STATE Break the inner loop.
            \ENDIF
            \STATE Set $\bfp_j\gets(1-\delta_j)\bfp_{j-1}+\delta_j\bfy_j$ for some $\delta_j\in [0,1]$.
        \ENDFOR
        \STATE Set: $\bfx_k\gets\bfp_j,d_k\gets\frac{\widehat{d}_{k-1}}{\rho},B_k\gets C_j$.
        \STATE Find $(\bar{\bfx}_k,\bar{d}_k)$ such that $S(\bar{\bfx}_k,\bar{d}_k)=S(\bfx_k,d_k)\cap S(\widehat{\bfx}_{k-1},\widehat{d}_{k-1})$ by using \eqref{Eq: intsect_simplex_balls}.
    \ENDFOR
\end{algorithmic}
\end{algorithm}

Alg.~\ref{Alg: rSFW} keeps three sequences
$
 \{ (\bfx_k, d_k)\}
$,
$
\{ (\bar{\bfx}_k, \bar{d}_k)\}
$, and
$
\{ (\widehat{\bfx}_k, \widehat{d}_k)\}
$, each associated with a simplex ball, namely $S(\bfx_k, d_k)$,
$S(\bar{\bfx}_k, \bar{d}_k)$, and $S(\widehat{\bfx}_k, \widehat{d}_k)$.
Starting with
$
(\bfx_0, d_0) = (\bar{\bfx}_0, \bar{d}_0)
$,
we define $(\widehat{\bfx}_0, \widehat{d}_0)$ such that $S(\widehat{\bfx}_0, \widehat{d}_0) = S(\bar{\bfx}_0, \bar{d}_0) \cap S_n$.
At the $k$th iteration, we first compute
\[
  \bfx_k \approx \argmin \left\{
   f(\bfy) \ | \ \bfy \in S(\widehat{\bfx}_{k-1}, \widehat{d}_{k-1})
  \right\} \quad \mbox{and} \quad
  d_k = \widehat{d}_{k-1}/\rho.
\]
We then define $(\bar{\bfx}_k, \bar{d}_k)$ by its simplex ball, which satisfies $S(\bar{\bfx}_k,\bar{d}_k)=S(\bfx_k,d_k)\cap S(\widehat{\bfx}_{k-1},\widehat{d}_{k-1})$.
We further define the iterate $(\widehat{\bfx}_k, \widehat{d}_k)$ by its simplex ball satisfying $S (\widehat{\bfx}_k, \widehat{d}_k) = S(\bar{\bfx}_k,\bar{d}_k) \cap S_n$.
They can all be efficiently computed via the formula \eqref{Eq: intsect_simplex_balls}.

A great advantage of Alg.~\ref{Alg: rSFW} is its computation of $\bfx_k$.
The oracle we call is SLMO-2, which ensures that the iteration complexity remains the same as the original FW algorithm.
It is also important to highlight that the inner loop of our algorithm (Lines 5-13) follows the standard FW algorithm.
Its primary goal is to find a solution $\bfp_j$ that satisfies $f(\bfp_j) - C_j \leq \frac{\mu}{2\rho^2} \widehat{d}_{k-1}^2$. Therefore, various speedup techniques for the classical FW algorithm can be directly applied to this inner loop without interfering with the outer loop of the algorithm. These include approaches such as the `away-step' and `pairwise' variants of FW proposed by \cite{lacoste2015global}, `fully-corrective' variant of FW proposed by \cite{jaggi2013revisiting},  as well as the warm start technique suggested by \cite{freund2016new}.

The following theorem summarizes the convergence result for this algorithm.

\begin{theorem}\label{Thm: Convergence-SL2}
    Let $\{\bfx_k\}$ be the sequences generated by Alg.~\ref{Alg: rSFW} with step size policy for $\{\delta_j\}$ in \eqref{Eq: step size_const}-\eqref{Eq: step size_smooth}. Then, for $k\geq 1$, we have
    \begin{equation*}
        f(\bfx_k) -f^*\leq f(\bfx_k)-B_k\leq \frac{\mu}{2n^2\rho^{2k}}.
    \end{equation*}
\end{theorem}
\begin{proof}
    We first claim that $\bfx^*\in S(\widehat{\bfx}_k,\widehat{d}_k)$ for any $k\geq 0$
    and we prove this by induction.
For $k=0$, we have $(\bar{\bfx}_0, \bar{d}_0) = (\bfx_0, d_0) = (\bfone_n/n, 1/n)$.
By the definition of $S(\widehat{\bfx}_0, \widehat{d}_0)$, we have
\[
  S(\widehat{\bfx}_0, \widehat{d}_0) = S(\bar{\bfx}_0, \bar{d}_0) \cap S_n = S(\bfone_n/n, 1/n) \cap S_n = S_n.
\]
Hence, $\bfx^* \in  S(\widehat{\bfx}_0, \widehat{d}_0)$.
  Now suppose that $\bfx^*\in S(\widehat{\bfx}_{k-1},\widehat{d}_{k-1})$ for some $k\geq 1$.
    Note that the inner loop of Alg.~\ref{Alg: rSFW} corresponds to the standard Frank-Wolfe algorithm.
    By Theorem \ref{Thm: FW} and Lemma \ref{Lemma: Simplex ball}(4), which implies that the diameter of
    $S(\widehat{\bfx}_{k-1},\widehat{d}_{k-1})$ is $\sqrt{2}n \widehat{d}_{k-1}$, we have
    \begin{equation*}
        f(\bfp_j)-f^*\leq \frac{2L}{j+1}\left(\sqrt{2}n\widehat{d}_{k-1}\right)^2 = \frac{4Ln^2\widehat{d}_{k-1}^2}{j+1}
    \end{equation*}
    hold for all $j\in [J]$. In the case where the inner loop terminates at $j = J$, we obtain
    \begin{equation*}
        f(\bfx_k)-f^*=f(\bfp_J)-f^*\leq f(\bfp_J)-C_J\leq \frac{2L}{\frac{8\rho^2n^2L}{\mu}}2n^2\widehat{d}_{k-1}^2 = \frac{\mu}{2\rho^2}\widehat{d}_{k-1}^2 = \frac{\mu}{2}d_{k}^2.
    \end{equation*}
    Similarly, if the inner loop is interrupted due to lines 9-11 of the algorithm, we still have
    $$
    f(\bfx_k)-f^*\leq f(\bfp_j)-C_j\leq \frac{\mu}{2\rho^2}\widehat{d}_{k-1}^2= \frac{\mu}{2}d_{k}^2.
    $$
    Using the fact that $f(\bfx_k)-f^*\geq \frac{\mu}{2}\lVert \bfx_k-\bfx^*\rVert^2$, we have $\lVert \bfx_k-\bfx^*\rVert^2\leq d_{k}^2$, which implies via Lemma \ref{Lemma: Simplex ball}(5) that $\bfx^*\in S({\bfx}_k,{d}_k)$. This implies
   \begin{eqnarray*}
     \bfx^* &\in&  S({\bfx}_k,{d}_k)  \cap S(\widehat{\bfx}_{k-1},\widehat{d}_{k-1})
       \qquad (\mbox{as} \ \bfx^* \in S(\widehat{\bfx}_{k-1},\widehat{d}_{k-1}) \ \mbox{by induction})\\
     &=& S({\bfx}_k,{d}_k)  \cap S(\widehat{\bfx}_{k-1},\widehat{d}_{k-1}) \cap S_n  \qquad (\mbox{as}\ \bfx^* \in S_n)\\
     &=& S(\bar{\bfx}_k, \bar{d}_k) \cap S_n
     = S(\widehat{\bfx}_k,\widehat{d}_k) .
   \end{eqnarray*}

    We now start to prove the conclusion in Theorem \ref{Thm: Convergence-SL2}. Since $d_0=\frac{1}{n}$ and $\widehat{d}_{k}\leq d_k=\frac{\widehat{d}_{k-1}}{\rho}\leq \frac{d_{k-1}}{\rho}$, we have $\widehat{d}_{k}\leq \frac{1}{n\rho^k}$ and thus
    \begin{equation*}
        f(\bfx_k)-f^*\leq f(\bfx_k)-B_k \leq \frac{\mu}{2}{d}_{k}^2 \leq\frac{\mu}{2n^2\rho^{2k}}.
    \end{equation*}
    We complete the proof.
\end{proof}

\begin{remark}\label{Remark-warm-start}
(Warm-start Strategy)
For rSFW and rSFW${}_\calP$ in the upcoming Section~\ref{Section: gen_P}, when using the simple step size, the initial steps of each inner loop may perform poorly, causing the iteration point $\bfp_j$ far away from the optimal solution.
We found that the following heuristic warm-start strategy is effective in practice.
Let $J_k \ll J$ denote the actual number of inner loop iterations during the $k$-th outer loop iteration.
When initiating the $(k+1)$-th outer loop, instead of starting the inner loop from $j = 1$, begin from either $j = \frac{J_k}{\rho'}$ or $j = \sqrt{\frac{\bar{d}{k}}{\bar{d}{k-1}}} J_k$.
Here $\rho'>1$ is a hyperparameter, with $\rho' = 2$ serving as a reasonable default value.
\end{remark}

\begin{remark}\label{Remark-Growth-Condition}
(Extension to Quadratic Growth Condition)
Although the two main Thms.~\ref{Thm: Convergence-SL1} and
\ref{Thm: Convergence-SL2} are established under the strong convexity of $f(\cdot)$,
we would like to point out that the assumption can be weakened to quadratic growth condition:
\[
  (f(\bfx) - f(\bfx^*))^{1/2} \ge c~\mbox{dist}(\bfx, X^*), \qquad \forall\
  \bfx \in \calP,
\]
where $c>0$ and $X^*$ is the solution set of Problem
\eqref{Problem: general}.
This includes the well-known case $f(\bfx)=g(A\bfx) + \langle \bfb, \bfx\rangle$ with
$g$ strongly convex, $A \in \mathbb{R}^{m\times n}$ and $\bfb \in \mathbb{R}^n$,
for a detailed investigation of this class of functions with FW methods, see
\cite{beck2004conditional}.
In this paper, we did not make effort for such extension as our main purpose
is to introduce SLMO under the strong convexity setting for the sake of simplicity.

\end{remark}

\section{Generalization to Arbitrary Polytopes}\label{Section: gen_P}

This section extends the previous results for the unit simplex to arbitrary polytopes. This generalization allows for a broader application of our findings, facilitating their relevance to a wider range of optimization problems.
Important properties between the standard simplex $S_n$ and general polytopes have been established in \cite{garber2016linearly}. Our extension heavily relies on some of those results.
This section is patterned after Section~\ref{Section-SFW} with some details omitted to avoid repeating.
We first define the simplex ball for general polytope and the corresponding
SLMO. We then describe the Simplex Frank-Wolfe for general polytope, followed by
its refined version.

\subsection{Simplex Ball and SLMO for Arbitrary Polytopes} \label{Subsection-SLMO-P}

Consider Problem \eqref{Problem: general} with $\calP=\Conv(\mathcal{V})$ and $\mathcal{V}=\{\bfv_1,\dots,\bfv_N\}$.
Therefore, any given $\bfx \in \calP$ can be represented as a convex combination
of those atoms $\bfv_i$. However, the convex combination may not be unique.
We define a set-valued mapping $\calM$ from $\calP$ to the following set:
\[
 \calM (\bfx) := \left\{
  \bflambda \in S_N \ \left|  \ \bfx = \sum_{i=1}^N \lambda_i \bfv_i  \right.
 \right\} .
\]
Recall that for $\bflambda \in \mathbb{R}^N$ and $d>0$, $S(\bflambda, d)$ is the
simplex ball defined in \eqref{SimplexBall}.
The idea of defining a similar Simplex ball over the polytope can be summarized as follows:
\be \label{SimplexBall-P}
 \bfx \in \calP \ \Longrightarrow \
 \bflambda_x \in \calM(\bfx) \ \Longrightarrow \
 S(\bflambda_x, d) \ \Longrightarrow \
 S_{\calP}(\bfx, d) := \left\{ V\bflambda \ | \ \bflambda \in S(\bflambda_x, d) \right\} ,
\ee
where $V$ consists of the columns $\bfv_i$, $i\in [N]$.
It seems that the Simplex ball $S_{\calP}(\bfx, d)$ depends on a particular choice
of $\bflambda_x \in \calM(\bfx)$.
The following result dismisses this dependence.

\begin{lemma} \label{Lemma-TwoSimplexBall}
	Given $\bfx\in\calP$ and $d>0$, let $\bflambda_x, \bflambda_x'
	\in \calM(\bfx)$. Then for any $\bflambda\in S(\bflambda_x,d)$, there exist $\bflambda'\in S(\bflambda_x',d)$ such that
	\begin{equation*}
		\sum_{i=1}^N\lambda(i)\bfv_i=\sum_{i=1}^N\lambda'(i)\bfv_i.
	\end{equation*}
\end{lemma}
\begin{proof}
By definition of $\calM(\bfx)$, we know
\be \label{Eq-lambda}
    \sum_{i=1}^N \lambda_x(i) \bfv_i = \sum_{i=1}^N \lambda_x'(i) \bfv_i.
\ee

It follows from $\bflambda_x' = \bflambda_x - \bflambda_x + \bflambda_x'$ and the definition of Simplex ball \eqref{SimplexBall-New} that
\be \label{Eq-S}
S(\bflambda_x',d)=S(\bflambda_x,d)-\bflambda_x+\bflambda_x'.
\ee
Let $\bflambda' :=\bflambda-\bflambda_x+\bflambda_x'$.
Since $\bflambda \in S(\bflambda_x, d)$, the identity \eqref{Eq-S} implies
$\bflambda'\in S(\bflambda_x',d)$. Moreover, we have
	\begin{equation*}
		\sum_{i=1}^N\lambda'(i)\bfv_i =
		\sum_{i=1}^N(\lambda(i)-\lambda_x(i)+\lambda_x'(i))\bfv_i
		= \sum_{i=1}^N\lambda(i)\bfv_i,
	\end{equation*}
where the last equation used \eqref{Eq-lambda}. This
completes the proof.
\end{proof}



Lemma~\ref{Lemma-TwoSimplexBall} ensures that the definition is independent of choice of $\bflambda_x \in \calM(\bfx)$. Hence, the definition is well defined.
Given a linear objective $\bfc\in\mathbb{R}^n$, we extend it to $\bfc_{ext}\in\mathbb{R}^N$ such that $c_{ext}(i)=\langle \bfv_i,\bfc\rangle$ for all $i\in [N]$. Consequently, the following equivalence holds:
\begin{equation*}
    \min_{\bfy\in\calP}\langle \bfy,\bfc\rangle = \min_{\bflambda\in S_N}\langle \bflambda,\bfc_{ext}\rangle.
\end{equation*}
Leveraging this equivalence, we define the generalized SLMO for $\calP$ as follows.

\begin{definition}[$\SLMO_\calP$: SLMO over $\calP$]\label{Def: gen_SLMO}
    Given a linear objective $\bfc\in\mathbb{R}^n$, radius $d>0$, a point $\bfx\in\calP$, and its corresponding  $\bflambda_x\in \calM(\bfx)$, a solution $\bfy^*\in \SLMO_{\calP}(\bfx,d,\bfc, \bflambda_x)$ is referred to as a {generalized simplex-based linear minimization oracle} if
    \[\bfy^*=\sum_{i=1}^N\lambda_i^*\bfv_i,\]
    where $\bflambda^*$ is an optimal solution to the following optimization problem
    \begin{equation}\label{Eq: SLMO_P}
        \begin{aligned}
            &\min\quad \langle \bflambda,\bfc_{ext}\rangle \\
            &\mbox{s.t. }\quad \bflambda\in S(\bflambda_x,d)\cap S_N.
        \end{aligned}
    \end{equation}
\end{definition}

We note that \eqref{Eq: SLMO_P} can be efficiently solved by Alg.~\ref{Alg: SLMO}. Consequently, $\SLMO_\calP$ can also be efficiently solved provided an element in $\calM(\bfx)$ can be cheaply obtained.
The detailed steps are outlined in Alg.~\ref{Alg: SLMO_P}.

\begin{algorithm}[H]
\footnotesize
	\renewcommand{\algorithmicrequire}{\textbf{Input:}}
	\renewcommand{\algorithmicensure}{\textbf{Output:}}
	\caption{$\SLMO_\calP(\bfx, d, \bfc, \bflambda_x)$}
	\label{Alg: SLMO_P}
	\begin{algorithmic}[1]
        \REQUIRE point $\bfx\in\calP$ with $\bflambda_x \in \calM(\bfx)$,
        linear objective $\bfc\in\mathbb{R}^n$, radius $d>0$.
        \STATE $\widehat{d}\gets\frac{\sum_{i=1}^N\min\{\lambda_x(i),d\}}{n+1}$
        \STATE $\widehat{\bflambda}\gets\bflambda_x-\min\{\bflambda_x,d\bfone_N\}+\widehat{d}\bfone_N$
        \STATE $\bfy_+\gets\sum_{i=1}^N(\widehat{\lambda}_i-\widehat{d})\bfv_i$
        \STATE $\bfv_{i^*} \gets \argmin_{\bfv\in \calP}\langle \bfv,\bfc\rangle$
        \STATE $\bfy^*\gets \bfy_++(n+1)\widehat{d}\bfv_{i^*}$
        \ENSURE $\bfy^*$.
	\end{algorithmic}
\end{algorithm}
One can observe that $\SLMO_\calP$-2---corresponding to Lines 4-5 in Alg.~\ref{Alg: SLMO_P} and serving as the oracle in our subsequent Alg.~\ref{Alg: rSFW-P}---requires only one more vector addition and one extra scalar-vector multiplication compared to the standard LMO.

We summarize the optimality of Alg.~\ref{Alg: SLMO_P} in the following result, which is direct consequence of
Lemma~\ref{Lemma-SLMO}

\begin{lemma} \label{Lemma-SLMOP}
	Algorithm $\SLMO_\calP(\bfx, d, \bfc, \bflambda_x)$ returns an optimal solution $\bfy^*$ for the problem:
\[
  \bfy^* \in \argmin \left\{  \langle \bfc, \bfy \rangle \ | \ \bfy \in S_{\calP} (\bfx, d) \cap \calP \right\}.
\]
\end{lemma}

\begin{remark} \label{Remark-SLMO-P}
(Carath\'{e}odory's Representation Assumption)
By Carath\'{e}odory's \\ Representation Theorem \cite[Thm. 17.1]{rockafellar1997convex}, for any point $\bfx\in\calP$, there exists $\bflambda_x \in \calM(\bfx)$ such that $|\mathcal{I}_+(\bflambda_x)| \leq n + 1$ where $\mathcal{I}_+(\bflambda_x):=\{i \in [N] \mid \lambda_x(i) > 0\}$.
As demonstrated in the illustrative examples in Supplement \ref{Section: app_Cara_rep}, this representation can be easily implemented for common types of $\calP$.
In this representation, the running time of $\SLMO_\calP$ does not explicitly depend on the number of vertices $N$, but rather on the natural dimension of $\calP$, that is, $n$.
For the analysis in the subsequent sections, we assume without loss of generality that the selected $\bflambda_x$ always satisfies $|\mathcal{I}_+(\bflambda_x)| \leq n + 1$.
\end{remark}



%
The following lemma demonstrates some useful properties of $S_{\calP}$ and  $\SLMO_{\calP}$, which can be regarded as a generalization of Lemma \ref{Lemma: Simplex ball}(5) and is crucial for proving the convergence of our algorithm in the next subsection. The detailed proof can be found in Supplement \ref{Section: app_prove_gen_P}.

\begin{lemma}\label{Lemma: SLMO_P}
    Given $\bfx\in\calP$, $d>0$ and $\bfy^*\in \SLMO_{\calP}(\bfx,d,\bfc,\bflambda_x)$, for any point $\bfy\in\calP$ satisfying $\|\bfx-\bfy\|\leq \frac{dD}{\eta}$, it follows that $\bfy\in S_{\calP}(\bfx,d)$ and $\langle \bfc,\bfy^* \rangle\leq \langle \bfc,\bfy \rangle$. Furthermore, we have $\|\bfx-\bfy^*\|\leq (n+1)dD$.
\end{lemma}

We also like to note that, though similar to LLOO, $\SLMO_{\calP}(\bfx,d,\bfc,\bflambda_x)$ is not exactly an LLOO because Lemma~\ref{Lemma: SLMO_P} only proves that $\mathbb{B}(\bfx, (D/\eta) d) \subseteq S_{\calP}(\bfx,d)$, not $\mathbb{B}(\bfx, d) \subseteq S_{\calP}(\bfx,d)$ which would be sufficient for the first
property of LLOO.
We note that $D/\eta = \xi/\psi$. Therefore, the condition $\xi/\psi\ge 1$ would be enough for $\SLMO_{\calP}(\bfx,d,\bfc,\bflambda_x)$ to be LLOO.

\subsection{\texorpdfstring{$\SFW_\calP$: Simplex Frank-Wolfe for Arbitrary Polytopes}{SFW-P: Simplex Frank-Wolfe for Arbitrary Polytopes}}

In this subsection, we extend the SFW to the polytope case. The generalized SFW algorithm is presented as follows.
\begin{algorithm}[!ht]
\footnotesize
\renewcommand{\algorithmicrequire}{\textbf{Input:}}
\renewcommand{\algorithmicensure}{\textbf{Output:}}
\caption{$\SFW_{\calP}$: Simplex Frank-Wolfe Method for Polytope $\calP$}
\label{Alg: SFW_P}
\begin{algorithmic}[1]
    \REQUIRE $\bfx_0\in S_n$, initial lower bound $B_0$, condition number $\eta$ and diameter $D$ of $\calP$.
    \STATE Set: $d_0\gets\sqrt{\frac{2(f(\bfx_0)-B_0)}{\mu}},\bflambda_0 \in \calM(\bfx_0).$
    \FOR{$k=1,\dots$}

        \STATE Compute $\bfy_k\in \SLMO_{\calP}({\bfx}_{k-1},\frac{\eta}{D}{d}_{k-1},\nabla f(\bfx_{k-1}),\bflambda_{k-1})$.

        \STATE Compute the working lower bound: $B_k^w\gets f(\bfx_{k-1})+\langle\nabla f(\bfx_{k-1}),\bfy_k-\bfx_{k-1}\rangle$.

        \STATE Update best bound $B_k\gets \max\{B_{k-1}, B_k^w \}$.

        \STATE Set $\bfx_k\gets (1-\delta_k)\bfx_{k-1}+\delta_k\bfy_k$ for some $\delta_k\in [0,1]$.


        \STATE Set: $d_k\gets \sqrt{\frac{2(f(\bfx_k)-B_k)}{\mu}},\bflambda_k \in \calM(\bfx_k)$.
    \ENDFOR
\end{algorithmic}
\end{algorithm}

Notice that when $\calP$ degenerates to $S_n$, the algorithm differs from Alg.~\ref{Alg: SFW_P} only slightly at line 7 since $\eta=D = \sqrt{2}$ for $S_n$.
The convergence for Alg.~\ref{Alg: SFW_P} stated below is proved in Supplement \ref{Section: app_prove_gen_P}.

\begin{theorem}\label{Thm: Convergence_sFW_P}
    Let $\{\bfx_k\}$ be the sequences generated by Alg.~\ref{Alg: SFW_P} with step size policy for $\{\delta_k\}$ in \eqref{Eq: step size_line}, \eqref{Eq: step size_smooth}, or simple step size
    \begin{equation}\label{Eq: step size_constant3}
        \delta_k=\frac{\mu}{2L(n+1)^2\eta^2}.
    \end{equation}
    Then, for $k\geq 0$, we have
    \begin{equation}\label{Eq: Convergence_sfw_P}
        f(\bfx_k) -f^*\leq f(\bfx_k) - B_k\leq \frac{\mu d_0^2}{2}e^{-\frac{\mu}{4L\eta^2(n+1)^2}k}.
    \end{equation}
\end{theorem}

\subsection{\texorpdfstring{$\mbox{rSFW}_\calP$: Refining $\SFW_\calP$}{rSFW-P: Refining SFW-P}}

As with the motivation for rSFW for the Simplex case, once we constructed the
Simplex ball for $\calP$, we may compute an approximate solution:
\[
  \bfp_k \approx \argmin \; f(\bfp) \quad \mbox{s.t.} \quad
  \bfp \in S(\bfx_{k-1}, d_{k-1}) \cap S_N ,
\]
with the initial point $\bfp_0 = \bfx_{k-1}$.
The motivation is based on a similar observation that $\SFW_\calP$ can be split into
two independent parts with the first part of constructing the Simplex ball being
the major computation. Hence, once such a ball is constructed we run a few more cheap $\SLMO$ steps over this ball.
Once again, other methods such as Away-step FW and pairwise FW can be used for computing $\bfp_k$.
The generalized rSFW algorithm is presented as follows.

\begin{algorithm}[!ht]
\footnotesize
\renewcommand{\algorithmicrequire}{\textbf{Input:}}
\renewcommand{\algorithmicensure}{\textbf{Output:}}
\caption{$\mbox{rSFW}_\calP$:  Refined Simplex Frank-Wolfe Method for Polytope $\calP$}
\label{Alg: rSFW-P}
\begin{algorithmic}[1]
    \REQUIRE $\bfx_0\in\calP$, $\bflambda_0 \in \calM(\bfx_0)$,
    radius contraction ratio $\rho>1$, initial lower bound $B_0$, condition number $\eta$ and diameter $D$ of $\calP$.
    \STATE Set: $d_0\gets\frac{\eta}{D}\sqrt{\frac{2(f(\bfx_0)-B_0)}{\mu}}, J\gets\frac{4\rho^2(n+1)^2\eta^2L}{\mu}$.
    \FOR{$k=1,\dots$}
        \STATE Set: $\bfp_0\gets\bfx_{k-1},C_0\gets B_{k-1}$.
        \STATE ($\SLMO_{\calP}$-1)
         Compute $\widehat{\bflambda}_{k-1}$ and $\widehat{d}_{k-1}$ such that $S(\widehat{\bflambda}_{k-1},\widehat{d}_{k-1})=S({{\bflambda}}_{k-1},{d}_{k-1})\cap S_N$.
        \FOR{$j=1,\dots,J$}
            \STATE ($\SLMO_{\calP}$-2)
            Compute $\bfy_j\in \SLMO_{\calP}(\bfx_{k-1},d_{k-1},\nabla f(\bfp_{j-1}),\bflambda_{k-1})$.
            \STATE Set: $C_j^w\gets f(\bfp_{j-1})+\langle \nabla f(\bfp_{j-1}),\bfy_j-\bfp_{j-1} \rangle$.
            \STATE  Update best bound $C_j\gets \max\{C_{j-1},C_j^w\}$.
            \IF{$f(\bfp_j)-C_j\leq \frac{\mu}{2\rho^2\eta^2}{d}_{k-1}^2D^2$}
                \STATE Break out of the inner loop.
            \ENDIF
            \STATE Set $\bfp_j\gets(1-\delta_j)\bfp_{j-1}+\delta_j\bfy_j$ for some $\delta_j\in [0,1]$.
        \ENDFOR
        \STATE Set: $\bfx_k\gets\bfp_j$, $d_k\gets\frac{{d}_{k-1}}{\rho}$,
        $B_k\gets C_j$ and $\bflambda_k \in \calM(\bfx_k)$.
    \ENDFOR
\end{algorithmic}
\end{algorithm}
Similar to Theorem \ref{Thm: Convergence-SL2}, we can provide the following convergence analysis for Alg.~\ref{Alg: rSFW-P}, which is proven in Supplement~\ref{Section: app_prove_gen_P}.
\begin{theorem}\label{Thm: Convergence_Sp_sFW_P}
    Let $\{\bfx_k\}$ be the sequences generated by Alg.~\ref{Alg: rSFW-P} with step size policy for $\{\delta_j\}$ in \eqref{Eq: step size_const}-\eqref{Eq: step size_smooth}. Then, for $k\geq 1$, we have
    \begin{equation*}
        f(\bfx_k) -f^*\leq f(\bfx_k)-B_k\leq (f(\bfx_0)-B_0)\rho^{-2k}.
    \end{equation*}
\end{theorem}

\begin{remark}
(Adaptive Lower Bound Update)
    We estimate the lower bound of $f^*$ by $f(\bfx_{k-1})+\langle\nabla f(\bfx_{k-1}),\bfy_k-\bfx_{k-1}\rangle$ for the Simplex Frank-Wolfe method and its refined version.
    In fact, when the objective function exhibits specific structural properties, we  can derive an additional lower $B_k^o$ and update the best bound $B_k$ as $B_k\gets \max\{B_{k-1}, B_k^w, B_k^o \}$.
    For instance, when the objective function has a minmax structure, we can construct  a minmax lower bound $B_k^o$ for $f^*$, see \cite{freund2016new} for detailed analysis.
    Moreover, in certain application scenarios, there may be exact information about the optimal value $f^*$, such as in linear regression or machine learning tasks, where it is known a priori that the optimal value of the loss function is $0$. In such cases, it is straightforward to set $B_k^o \gets f^*$.
\end{remark}

\begin{remark}
(Robustness to Parameter Estimation)
Our algorithms rely on parameters $L, \mu, \eta, D$. In practice, using overestimates $L', \eta', D'$ and an underestimate $\mu'$ such that $\frac{L'\eta'D'\mu}{L\eta D\mu'} = O(1)$ only increases the bounds by a constant factor.
Moreover, as shown in Supplement~\ref{Section: app_quantity}, both $\eta$ and $D$ can be efficiently estimated for common $\calP$.
For $L$ and $\mu$, one can use the backtracking strategy from \cite{pedregosa2020linearly} to estimate their local values and compute adaptive short step sizes; see Subsection~\ref{Subsection-rSFW-A} for details.
\end{remark}

\section{Numerical Experiments} \label{Section-Numerical}

In this section, we present numerical experiments to evaluate the efficiency, convergence, and adaptability of the proposed methods. All tests were performed using MATLAB R2022b on a Windows laptop equipped with a 14-core Intel(R) Core(TM) 2.30GHz CPU and 16GB of RAM.

We try to furnish four tasks. 
(T1) We first assess the computational efficiency of SLMO and SLMO-2 across four representative polytopes,
consolidating their role of the workhorse in our SFW methods.
(T2) We illustrate the linear convergence behavior of SFW and rSFW using two numerical experiments.
(T3) We show that our methods can be enhanced with a backtracking strategy to eliminate the need for predefined values of the parameters $L$ and $\mu$. 
(T4) We demonstrate how integrating the away-step variants of the Frank-Wolfe method (AFW and PFW) into the rSFW framework significantly enhances its performance, outperforming the original AFW and PFW methods.
Those four tasks are addressed in four subsections.

\subsection{Efficiency of SLMO and SLMO-2}
In this subsection, we evaluate the performance of the proposed SLMO and SLMO-2 through comparative experiments on four common polytopes $\mathcal{P}$: (a) Unit simplex; (b) Hypercube; (c) $\ell_1$-ball; and (d) Flow polytope, derived from the video co-localization problem in \cite{joulin2014colocalization}.

\begin{table}[htbp]
\footnotesize
\caption{Description of projection and five LMO variants used in the numerical comparison. These six methods shares the same randomly generated parameters: $\bfc\sim \mathcal{N}(\bfzero,I_n),\bfx\in\calP$ and $d\in \mathcal{U}_{[0,1]}$.}\label{table: LMOs}
\begin{center}
    \begin{tabular}{p{0.13\linewidth}  p{0.30\linewidth} p{0.47\linewidth}} \hline
    Algorithm & Formulation & Description \\ \midrule[1.3pt]
    Projection & $\argmin_{\bfy\in\calP}\|\bfy-\bfz\|^2$ & \hangindent=1em \hangafter=1 The projection onto the polytope $\calP$, and $\bfz\sim \mathcal{N}(\bfzero,I_n)$ is a randomly generated point. We implement projections onto the Simplex and $\ell_1$-ball using the method from \cite[Fig.~2]{condat2016fast}, while the projection onto the hypercube is straightforward. Although a closed-form solution exists for projection onto the flow polytope \cite[Thm.~20]{vegh2012strongly}, its computational complexity of $O(m^3n + n^2)$ makes it significantly more expensive than other LMO variants. \\ \\
    LMO & $\argmin_{\bfy\in\calP}\langle \bfy,\bfc\rangle$ & \hangindent=1em \hangafter=1 The standard linear minimization oracle. \\ \\
    $\ell_1$-LMO & $$\argmin_{\bfy\in\{V\bflambda\mid\bflambda\in B_1(\bflambda_x,d)\cap S_N\}}\langle \bfy,\bfc\rangle$$ & \hangindent=1em \hangafter=1 The $\ell_1$-norm constrained LMO, Alg. 3 and Alg. 4 in \cite{garber2016linearly}. Here, $V$ consists of columns of $\bfv\in\calV(\calP)$, and the computation of  $\bflambda_x\in\calM(\bfx)$ is included in the timing. \\ 
    NEP & $$\argmin_{\bfy\in \calV (\calP)}\langle \bfy,\bfc\rangle + \lambda\|\bfy-\bfx\|^2$$ & \hangindent=1em \hangafter=1 Nearest extreme point oracle in \cite{garber2021frank}. Here, $\lambda\sim \mathcal{U}_{[0,10000]}$ is a randomly generated positive number.   \\ 
    SLMO$_{\calP}$ & $$\argmin_{\bfy\in\{V\bflambda\mid\bflambda\in S(\bflambda_x,d)\cap S_N\}}\langle \bfy,\bfc\rangle$$ & \hangindent=1em \hangafter=1 Our proposed Simplex Linear Minimization Oracle (Alg.~\ref{Alg: SLMO} and Alg.~\ref{Alg: SLMO_P}). Here, the computation of $\bflambda_x\in\calM(\bfx)$ is included in the timing.\\ 
    SLMO${}_\calP$-2 & $$\argmin_{\bfy\in\{V\bflambda\mid\bflambda\in S(\widehat{\bflambda}_x,\widehat{d})\}}\langle \bfy,\bfc\rangle$$ & \hangindent=1em \hangafter=1 The latter phase of SLMO, consisting of Lines 4-5 of Alg.~\ref{Alg: SLMO} and Alg.~\ref{Alg: SLMO_P}. Here, $S(\widehat{\bflambda}_x,\widehat{d})=S(\bflambda_x,d)\cap S_N$ is precomputed and not included in the timing.  \\  \hline
    \end{tabular}
\end{center}
\end{table}

\begin{figure}[!ht]
    \centering
    \subfloat[]{
        \includegraphics[width=0.45\textwidth]{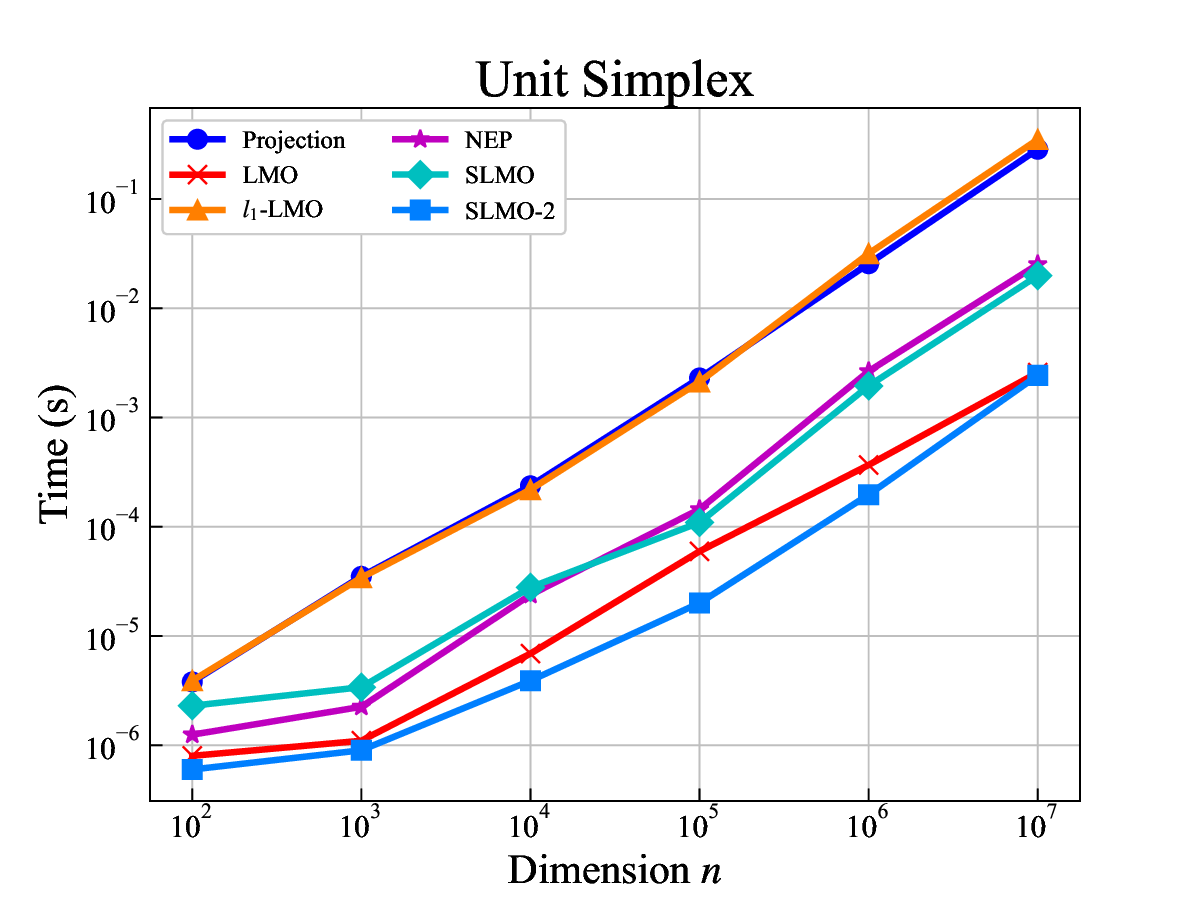}
        \label{Fig: LMO_compare_Simplex}
    }
    ~
    \subfloat[]{
        \includegraphics[width=0.45\textwidth]{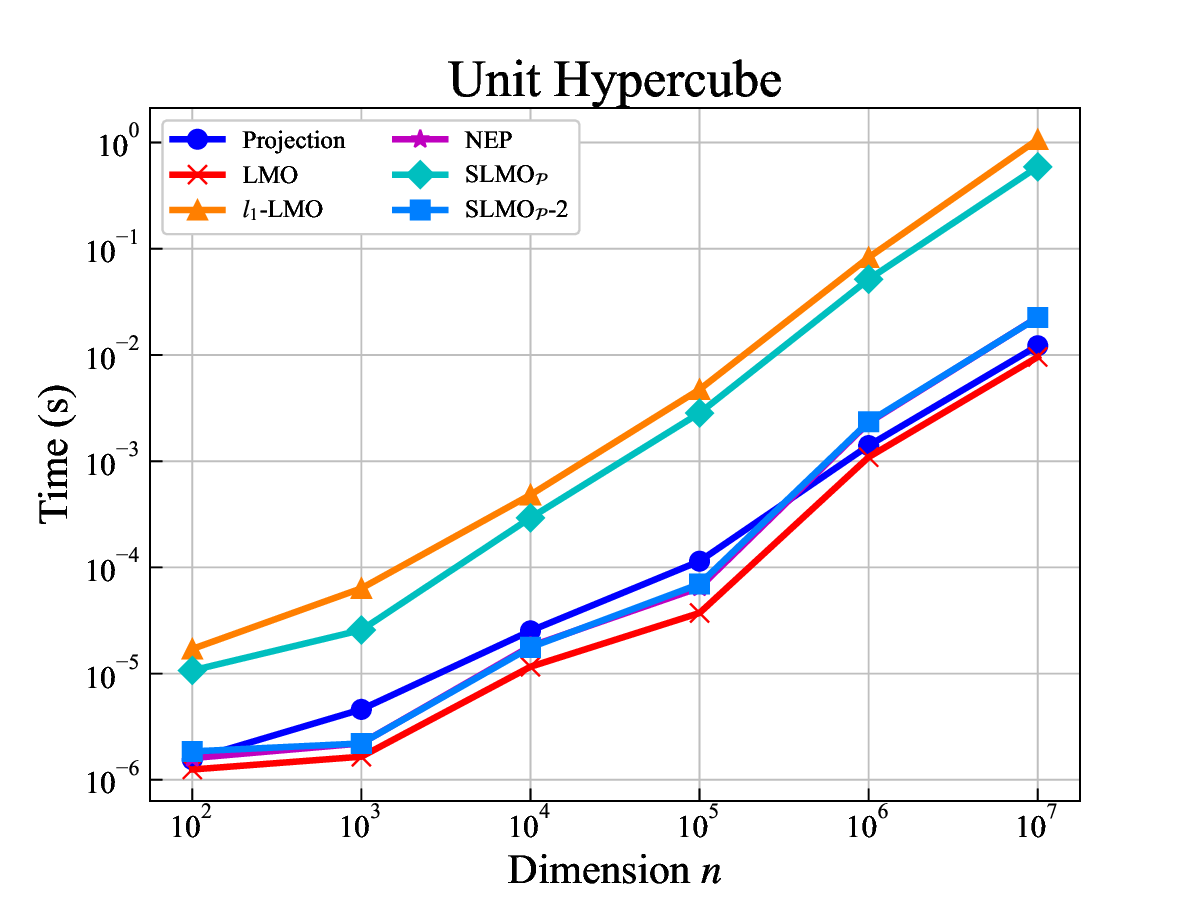}
        \label{Fig: LMO_compare_Hypercube}
    }

    \vspace{0.3cm} 

    \subfloat[]{
        \includegraphics[width=0.45\textwidth]{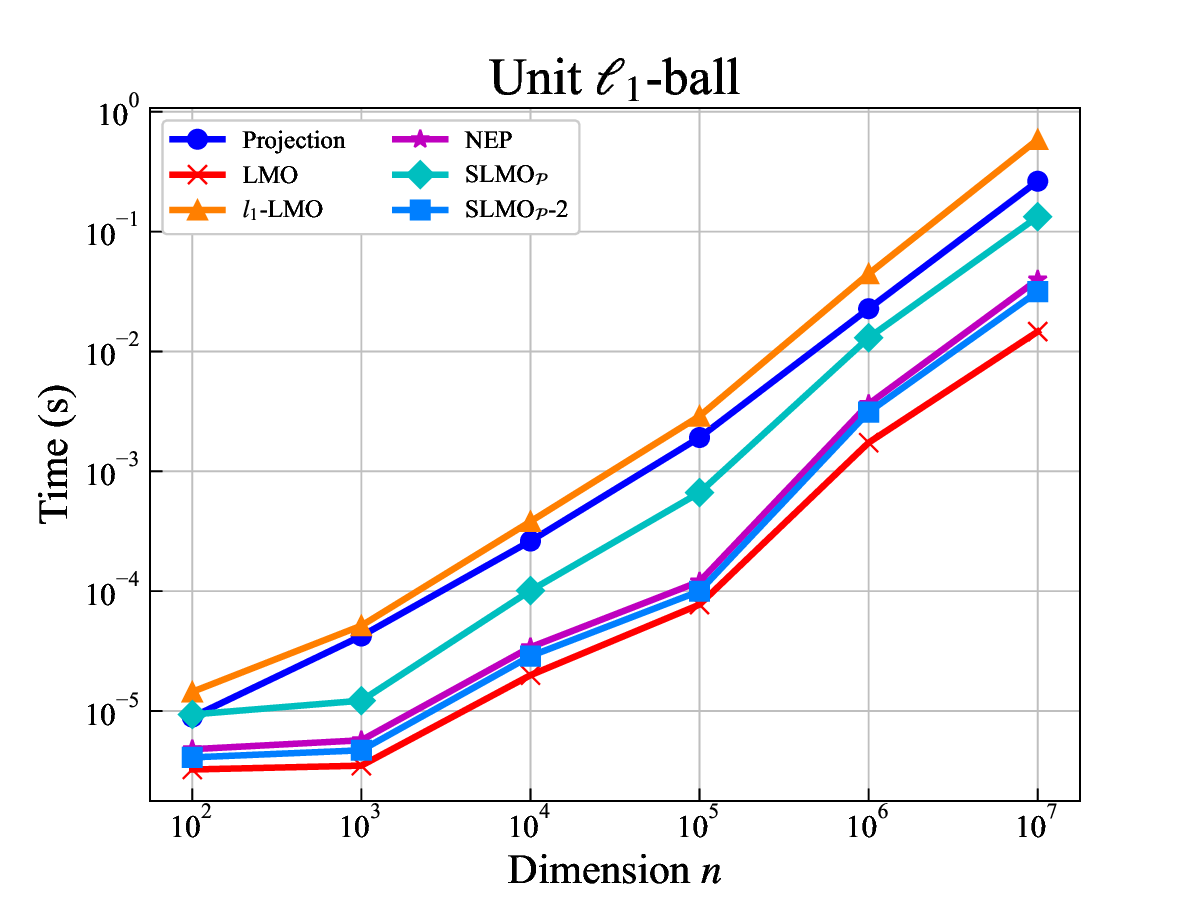}
        \label{Fig: LMO_compare_L1}
    }
    ~
    \subfloat[]{
        \includegraphics[width=0.45\textwidth]{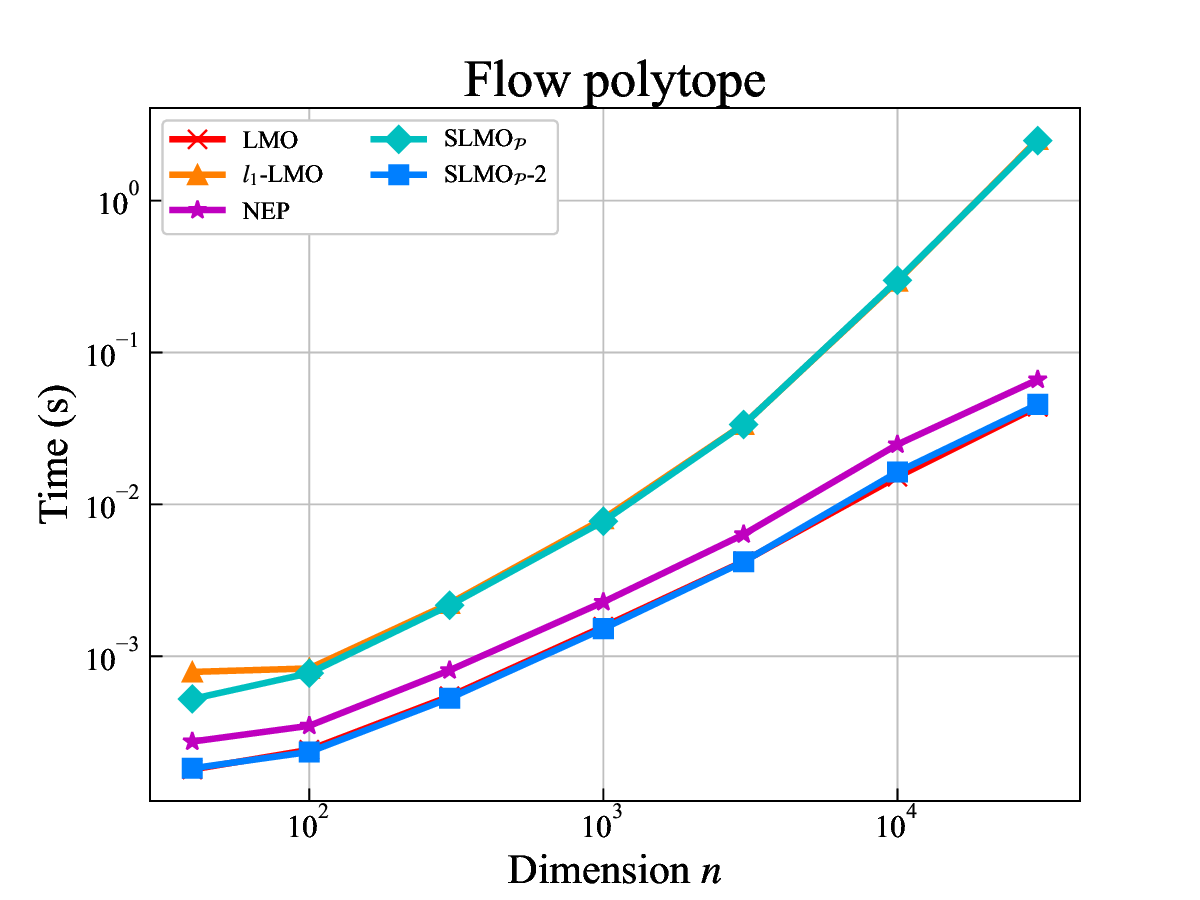}
        \label{Fig: LMO_compare_Flow}
    }
    \caption{Comparison of solving Projection, LMO, $\ell_1$-LMO, NEP, SLMO, and SLMO-2 over the following polytopes: (a) Unit Simplex; (b) Unit Hypercube; (c) Unit $\ell_1$-ball; and (d) Flow polytope. All the results are averaged over 20 i.i.d runs. We omit the projection onto the flow polytope due to its prohibitively high computational cost.}
    \label{Fig: LMO_compare}
\end{figure}

We consider the six different methods, including projection (a key subproblem in projection/proximal based methods) and five variants of LMO, as detailed in Table~\ref{table: LMOs}.
These six methods shares the same randomly generated $\bfc\sim \mathcal{N}(\bfzero,I_n),\bfx\in\calP$ and $d\in \mathcal{U}_{[0,1]}$.
Figure \ref{Fig: LMO_compare} illustrates the relationship between running time and dimensionality for the six methods. 
Additionally, Table~\ref{table: percentage-LMO} reports the proportion of time spent on LMO calls within the SLMO-2 algorithm.
We draw the following observations from these results:
\begin{itemize}
    \item \textbf{$\ell_1$-LMO:} Across all four polytopes, the $\ell_1$-LMO method incurs the highest computational overhead  surpassing even that of projection-based methods.
    \item \textbf{SLMO${}_\calP$-2 vs. SLMO${}_\calP$:} The overhead of SLMO-2 is significantly lower than that of SLMO. In most cases, as shown in Table~\ref{table: percentage-LMO}, its overhead closely matches that of the LMO itself. This indicates that our proposed rSFW and rSFW${}_\calP$ achieve iterative complexity comparable to that of the standard Frank-Wolfe algorithm.
    \item \textbf{NEP:} While NEP demonstrates very low runtime overhead, it is important to note that the corresponding Frank-Wolfe variant, NEP-FW, converges only sublinearly  as shown in Subsection~\ref{Subsection: num_sFW}.
    \item \textbf{$\ell_1$-LMO and SLMO${}_\calP$ on the Flow Polytope:} 
    Both methods exhibit rising overhead with increasing dimension, mainly due to the cost of computing the Carathéodory representation $\bflambda_x \in \calM(\bfx)$, which dominates the runtime.
\end{itemize}

\begin{table*}[!ht]
\footnotesize
\caption{Time overhead of LMO calls as a percentage of total computation time when using the SLMO-2 algorithm across four different polytopes.}
    \label{table: percentage-LMO}
\begin{center}
    \begin{tabular}{p{0.165\linewidth} p{0.165\linewidth} p{0.165\linewidth} p{0.165\linewidth} p{0.165\linewidth}} \hline
     & Simplex  & Hypercube  & $\ell_1$-ball  & Flow polytope  \\
     & ($n=10^7$) & ($n=10^7$) & ($n=10^7$) & ($n=3\times 10^4$) \\ \midrule[1.3pt]
    $\frac{\mbox{Time(LMO)}}{\mbox{Time(SLMO-2)}}$ & $98.8\%$ & $46.0\%$ & $52.6\%$ & $99.7\%$
    \\  \hline
    \end{tabular}

\end{center}
\end{table*}

\subsection{Linear Convergence of SFW and rSFW}\label{Subsection: num_sFW}
\begin{table*}[t]
{\footnotesize
\caption{Description of Frank-Wolfe variants used in the numerical comparison. }
    \label{table:algorithms}
\begin{center}
    \begin{tabular}{p{0.29\linewidth}  p{0.63\linewidth}} \hline
    Algorithm &Description \\ \midrule[1.3pt]
    FW (simple/Ada) & \hangindent=1em \hangafter=1 Frank-Wolfe  with simple step size $\delta_k = 2/(k+1)$. The `Ada' variant employs a backtracking step \eqref{Eq-backtracking} prior to updating the iterate, in order to estimate the local parameters $L$ and $\mu$, thereby enabling an adaptive short step size.
    We set $\tau_1 = 2$ and $\tau_2 = 0.9$ in Alg.~\ref{Alg: backtracking}, and apply the same configuration to the subsequent `Ada' variants.  \\
    SFW/SFW${}_\calP$ (line-search) &  \hangindent=1em \hangafter=1  Simplex Frank-Wolfe with exact line-search (Alg.~\ref{Alg: SFW} and Alg.~\ref{Alg: SFW_P}). For the $\ell_1$-constrained least squares problem, we set $\mu=2\lambda_{min}(A'A)$, $D=2$ and $\eta=\sqrt{n}$. Although Supplement~\ref{Section: app_quantity} estimates $\eta\leq n$ for $\ell_1$-ball, this setting does not hinder the algorithm  s linear convergence and demonstrates strong practical performance. For the video co-localization task, we set $\mu=\lambda_{min}(A)$ and $\eta = D=\sqrt{66}$; see Supplement~\ref{Section: app_quantity} for details. For the Simplex-constrained least squared problem, we set $\mu=2\lambda_{min}(A'A)$. \\
    NEP-FW (simple) & \hangindent=1em \hangafter=1 Frank-Wolfe with Nearest Extreme Point Oracle, with theoretical step size $2/(k+1)$ \cite[Alg.~1]{garber2021frank}. We omitted Line 5, as it showed no noticeable effect on performance. \\
    rSFW/rSFW${}_\calP$ (simple) & \hangindent=1em \hangafter=1 Refined Simplex Frank-Wolfe with simple step size $\delta_j = 2/(j+1)$ (Alg.~\ref{Alg: rSFW} and Alg.~\ref{Alg: rSFW-P}). We utilize the warm-start strategy with $\rho'=2$ as mentioned in Remark.~\ref{Remark-warm-start}. The parameters $\mu,D$ and $\eta$ are set the same as in SFW/SFW${}_\calP$. Additionally, we set $\rho=1.01, L=2\lambda_{max}(A'A)$ for $\ell_1$/Simplex-constrained least squares problems, and $\rho=1.01, L=\lambda_{max}(A)$ for video co-localization problem. \\
    PFW (line-search) & \hangindent=1em \hangafter=1 Pairwise Frank-Wolfe with exact line-search \cite[Alg.~2]{lacoste2015global}.  \\
    AFW (line-search) & \hangindent=1em \hangafter=1 Away-steps Frank-Wolfe with exact line-search \cite[Alg.~1]{lacoste2015global}. \\
    rSFW-P (line-search) & \hangindent=1em \hangafter=1 The Refined Simplex Frank-Wolfe framework enhanced with Pairwise technique. Specifically, we incorporate \eqref{Eq: PFW-correction} after Line 6 in Alg.~\ref{Alg: rSFW} and replace Line 12 with \eqref{Eq: rSFW_P_update_2}.\\
    rSFW-A (line-search) & \hangindent=1em \hangafter=1 The Refined Simplex Frank-Wolfe framework enhanced with Away-steps technique. Specifically, we incorporate \eqref{Eq: AFW-correction} after Line 6 in Alg.~\ref{Alg: rSFW} and replace Line 12 with \eqref{Eq: rSFW_P_update_2}.\\ \hline
    \end{tabular}

\end{center}
}
\end{table*}

We demonstrate the linear convergence of our proposed methods---SFW and rSFW through two numerical experiments. 
These methods are compared against the standard Frank-Wolfe (FW) algorithm, its variant NEP-FW\footnote{As the code for NEP-FW is not publicly available, we implemented it ourselves.}, and two well-known variants: Away-step FW\footnote{The implementations of AFW and PFW are available at \url{https://github.com/Simon-Lacoste-Julien/linearFW}.} (AFW) and Pairwise FW (PFW), all summarized in Table~\ref{table:algorithms}.

To evaluate algorithmic performance, we adopt the Frank-Wolfe gap defined by
$
\langle \nabla f(\bfx_k), \bfx_k-\bfy_{k+1}\rangle,
$
where $\bfx_k$ is the $k$-th iterate and $\bfy_k$ denotes the solution returned by the respective LMO variant at that iteration.
This FW gap provides a valid upper bound on the primal gap, i.e., $f(\bfx_k) - f^* \leq \langle \nabla f(\bfx_k), \bfx_k - \bfy_{k+1} \rangle$, and can thus be used as a practical stopping criterion\footnote{For NEP-FW, however, this inequality does not hold in general for $\bfy_{k+1} = \text{NEP}(\bfx_k) = \text{LMO}(\nabla f(\bfx_k) - \lambda_k \bfx_k, \calP)$.
Therefore, to ensure a consistent and fair comparison, we use the standard FW gap to evaluate NEP-FW as well.}.
\begin{figure}[!ht]
    \centering
    \subfloat[]{
        \includegraphics[width=0.45\textwidth]{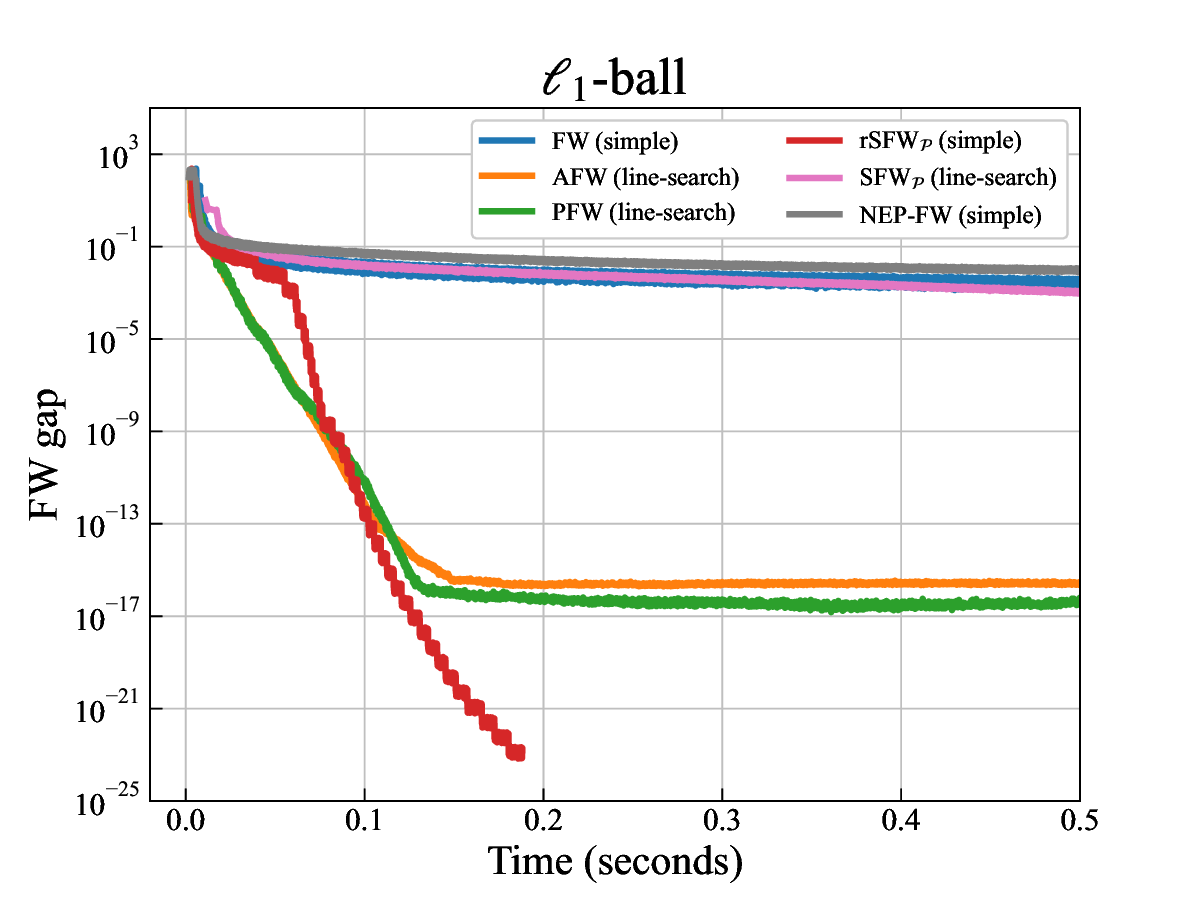}

        \label{Fig: time2gap_L1}
    }
    ~
    \subfloat[]{
        \includegraphics[width=0.45\textwidth]{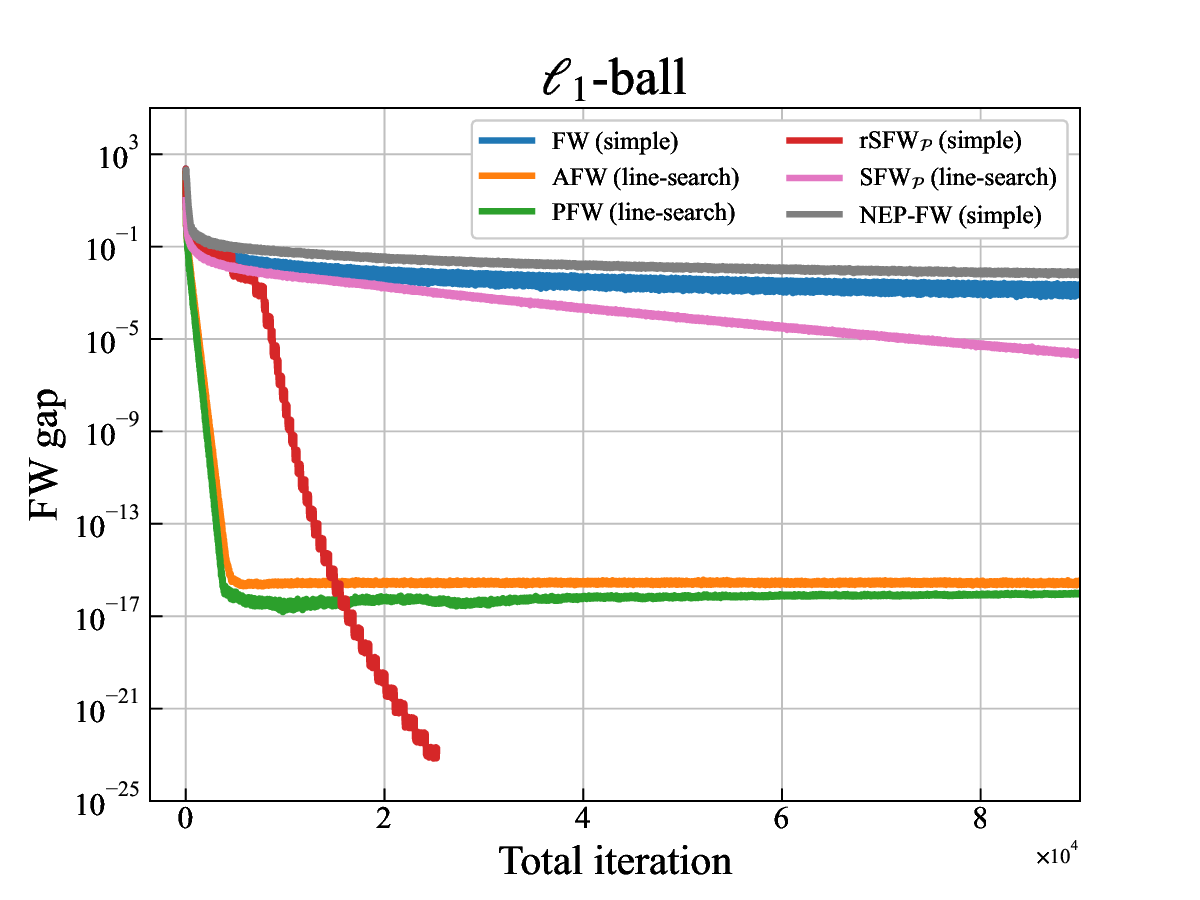}

        \label{Fig: iter2gap_L1}
    }

    \vspace{0.2ex} 

    \subfloat[]{
        \includegraphics[width=0.45\textwidth]{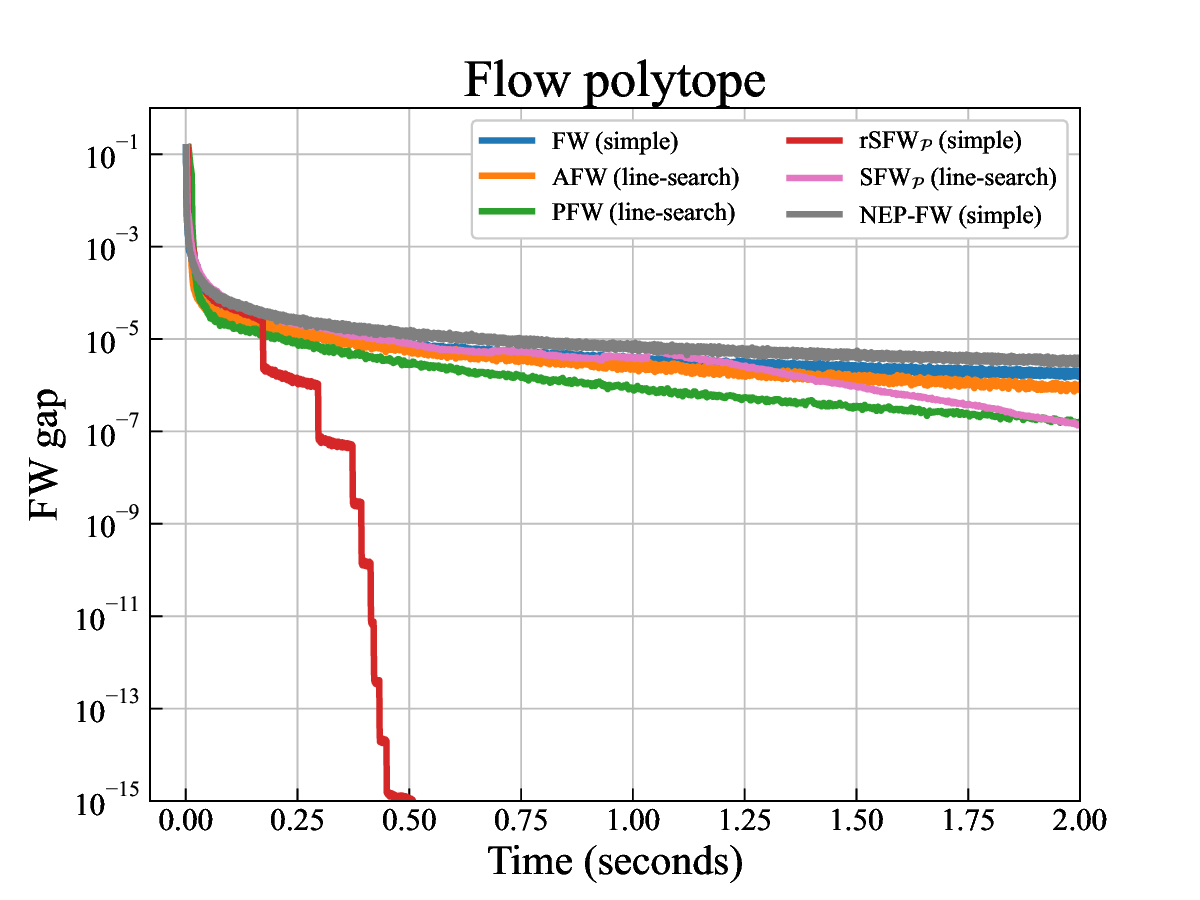}

        \label{Fig: time2f_Flow}
    }
    ~
    \subfloat[]{
        \includegraphics[width=0.45\textwidth]{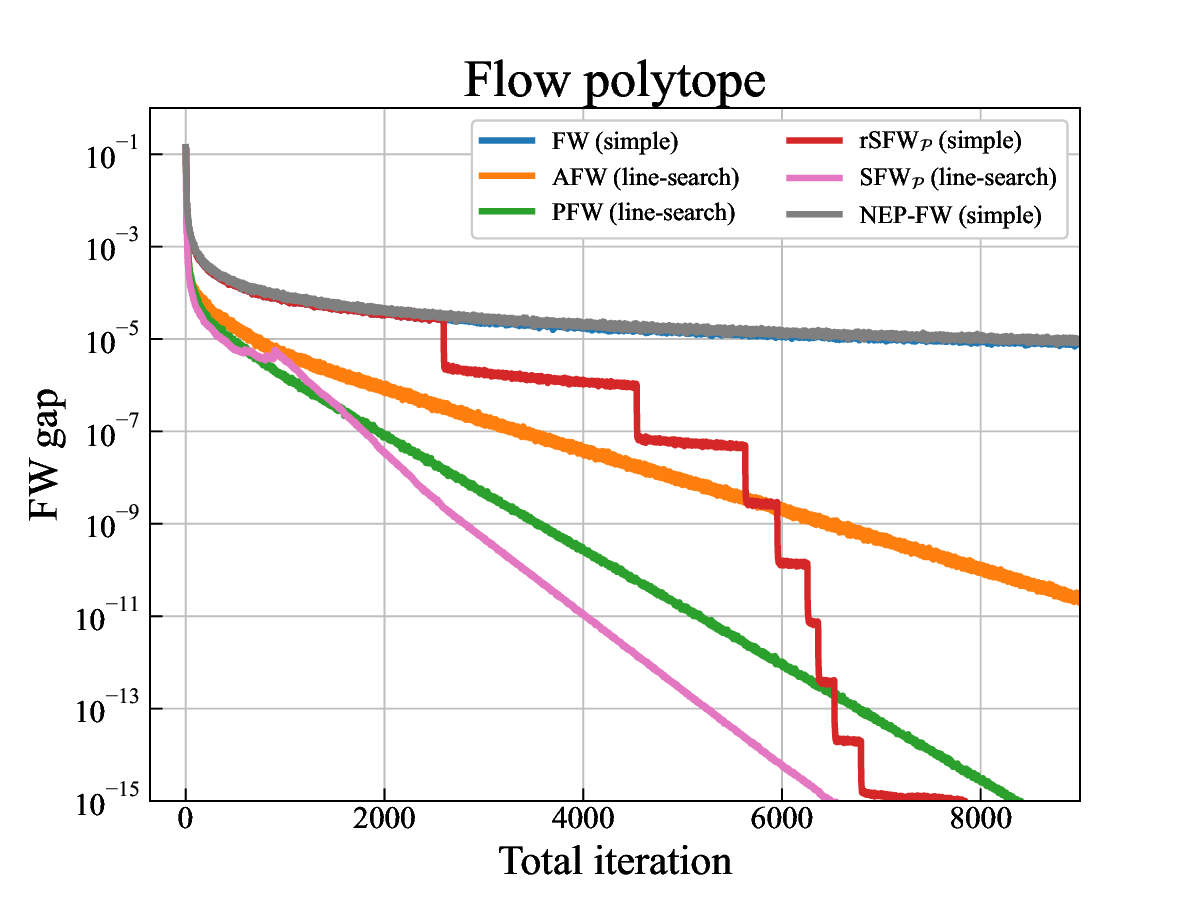}

        \label{Fig: iter2f_Flow}
    }

    \vspace{0.2ex}

    \caption{FW gap vs time/iterations.}
    \label{Fig: sFW_linear_demon}
\end{figure}

The first experiment involves an $\ell_1$-regularized least squares regression, that is $\min_{\|\bfx\|_1\leq 1}\lVert A\bfx-\bfb\rVert_2^2$, where $A\in\mathbb{R}^{m\times n}$ with $m=400,n=100$, and the entries of $A$ are drawn from a standard Gaussian distribution.
We set $\bfb=A\bfx^*$, where $\bfx^*$ is constructed by first generating a random vector with sparsity parameter $s=0.7$, followed by normalization to lie on the boundary of the $\ell_1$-ball.
Thus, the optimal value of this problem is 0. We use the same initial point $\bfx_0=\bfzero_n$ for all methods.

The second experiment involves a convex quadratic problem over the flow polytope, derived from the \textit{video co-localization} task introduced by \cite{joulin2014colocalization}.
The problem is formulated as $\min_{\bfx\in\mathcal{F}_{s,t}} \frac{1}{2}\bfx'A\bfx+\bfb' x$, where $A\in\mathbb{R}^{n\times n}$ is a positive definite matrix,  $\bfb\in\mathbb{R}^n$, and $\mathcal{F}_{s,t}$ represents the s-t flow polytope.
We used the same dataset and initial point as in \cite{lacoste2015global,garber2021frank}.
The problem has a dimension of $n = 660$.

The results are presented in Figure~\ref{Fig: sFW_linear_demon}. 
We make some comments below.
\begin{itemize}
    \item \textbf{Linear convergence of SFW${}_\calP$ and rSFW${}_\calP$:} Both SFW${}\calP$ and rSFW${}\calP$ show linear convergence, confirming our theoretical guarantees.
    \item \textbf{Superior efficiency of rSFW${}_\calP$:} In both experiments, our proposed rSFW method significantly outperforms all other algorithms  --including the well-established AFW and PFW--  in terms of running time.
    \item \textbf{Limitations of NEP-FW:} Although NEP performs well in iteration complexity, its NEP-FW variant converges sublinearly and is slightly slower than standard FW.
    \item \textbf{Time inefficiency in video co-localization:} In the video co-localization task, while SFW${}\calP$, AFW, and PFW converge quickly by iteration count, their runtime is slower due to overhead—--Carathéodory computation for SFW${}\calP$, and growing active sets for AFW and PFW.
\end{itemize}

\subsection{SFW/rSFW with Backtracking}
We further show that our methods can be enhanced with the backtracking technique proposed by \cite{pedregosa2020linearly}, thereby eliminating the need to manually specify the parameters $L$ and $\mu$.
While \cite[Alg.~2]{pedregosa2020linearly} does not specify how to estimate the strong convexity constant $\mu$, we outline our approach to estimating both $L$ and $\mu$ as well as determining an adaptive step size; see \ref{app-backtracking} for the detailed algorithm.
As an example, consider the $k$-th iteration of SFW${}_\calP$. Before updating $\bfx_k$, we perform the following backtracking step:
\begin{equation}\label{Eq-backtracking}
    \delta_k,\ L_k,\ \mu_k\gets \text{Backtracking-Routine}(\bfx_{k-1},\bfy_k-\bfx_{k-1},L_{k-1},\mu_{k-1},1).
\end{equation}

We focus on the $\ell_1$-constrained logistic regression problem with the form:
\[
    \min_{\lVert\bfx\rVert_1\leq \beta} \frac{1}{m}\sum_{i=1}^m \text{ln}(1+\text{exp}(-b_i\langle \mathbf{a}_i,\bfx \rangle))+\frac{\lambda}{2}\lVert\bfx\rVert^2,
\]
where $A=[\mathbf{a}_1,\dots,\mathbf{a}_m]\in\mathbb{R}^{m\times n}$ and $\bfb\in\mathbb{R}^m$. We use the dataset Madelon \cite{guyon2008feature}, which has $m=4400,n=500$, and fully-density (i.e. $\text{density}=1$). We set $\beta=1$ and $\lambda=1/n$.
We compare our methods against AFW, PFW, and the standard FW, all of which are equipped with the backtracking technique. The results are presented in Figure~\ref{Fig: FW_back}. It can be observed that our two methods achieve the best performance in terms of running time. Although AFW and PFW perform well in terms of iteration count, their overall efficiency is hindered by the increasing cost of maintaining a growing active set and computing the away direction.

\begin{figure}[ht]
    \centering
    \subfloat{
        \includegraphics[width=0.45\textwidth]{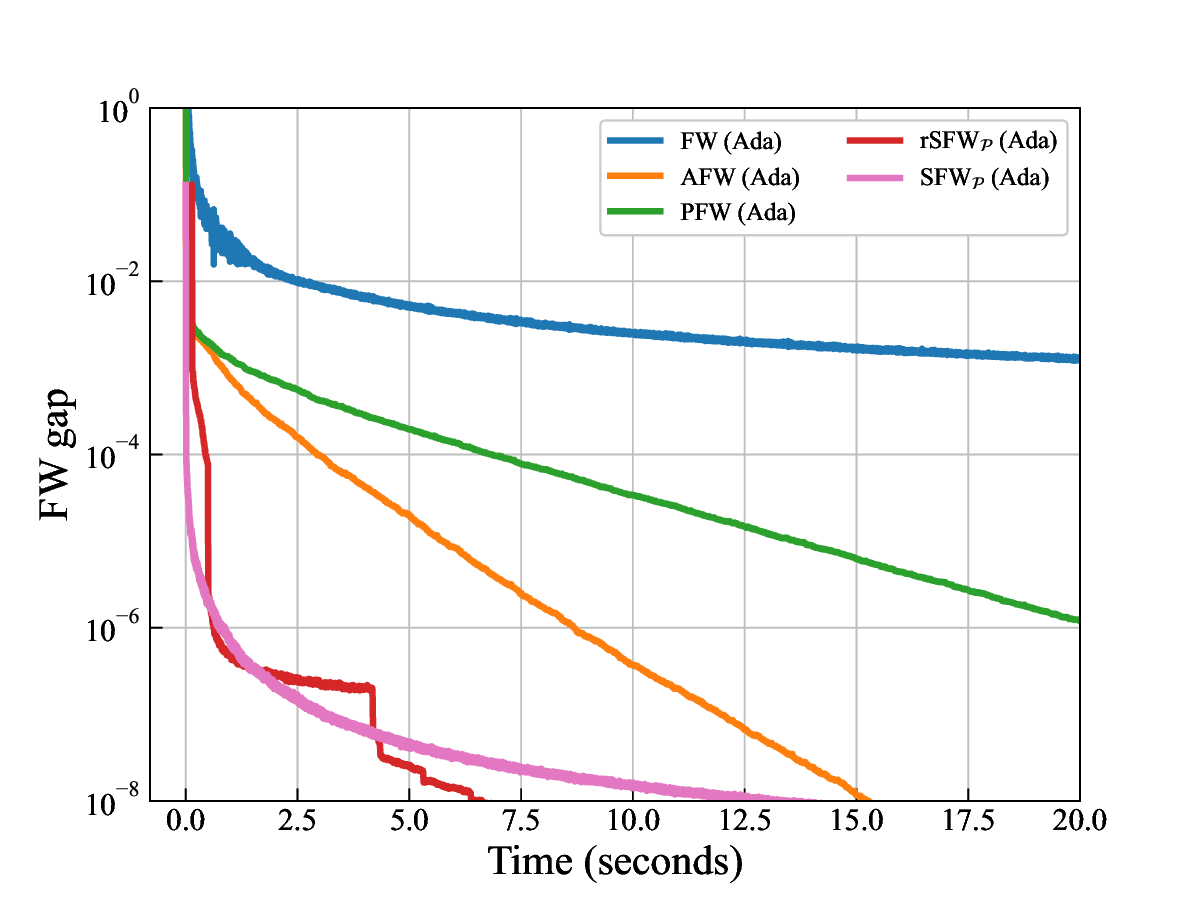}
        \label{Fig: time2gap_back}
    }
    ~
    \subfloat{
        \includegraphics[width=0.45\textwidth]{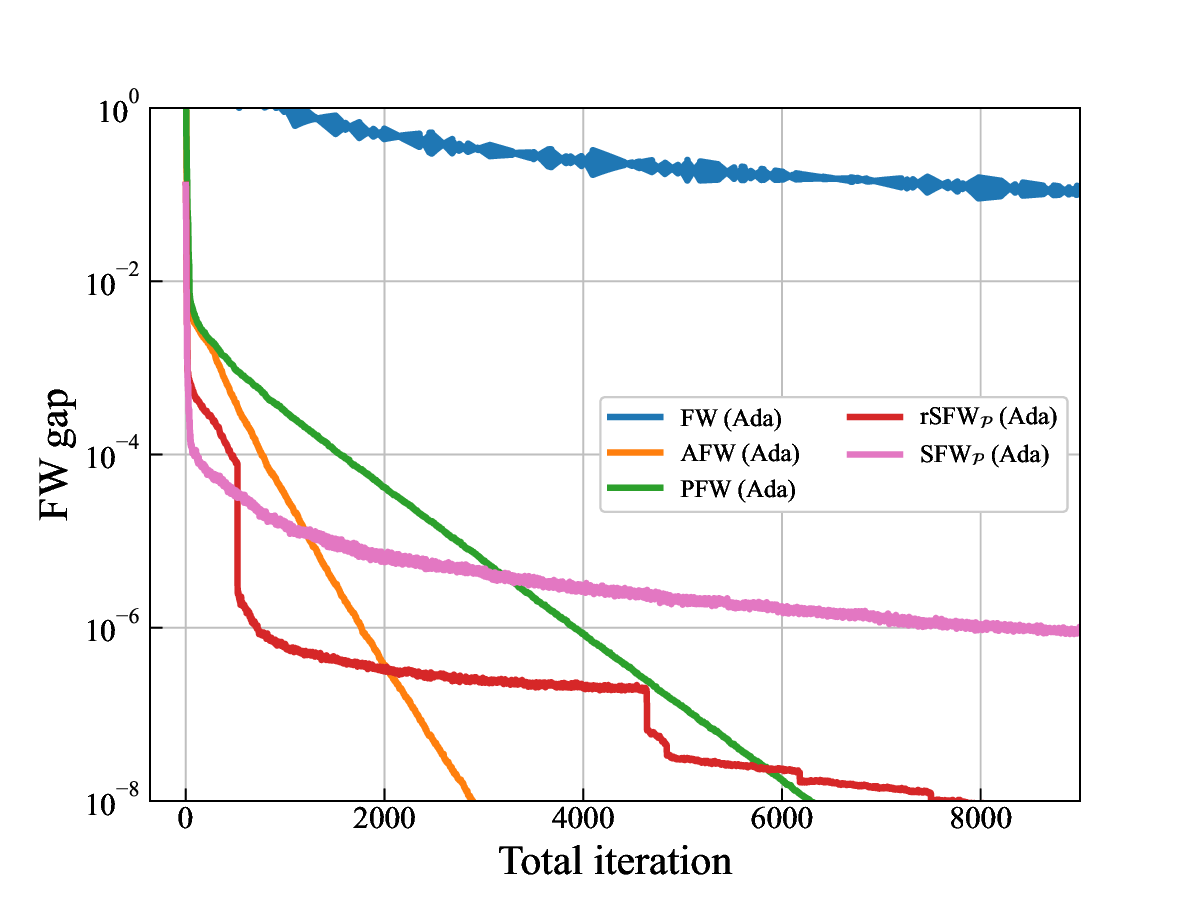}
        \label{Fig: iter2gap_back}
    }
    \caption{FW gap vs time/iterations on the $\ell_1$-constrained logistic regression problem.}\label{Fig: FW_back}
\end{figure}

\subsection{rSFW Framework Combined with the Away Step Technique}\label{Subsection-rSFW-A}

In this subsection, we demonstrate that the well-known linearly converging variants of the standard Frank-Wolfe method  AFW and PFW \cite{lacoste2015global}  can be seamlessly integrated into the inner loop of the rSFW framework.
This straightforward combination leads to a significant performance improvement over the original AFW and PFW methods.

We focus on the simplex-regularized problem $\min_{\bfx\in S_n}\lVert A\bfx-\bfb\rVert_2^2$, where $A\in\mathbb{R}^{m\times n},m=800,n=200$, with standard Gaussian entries.
We set $\bfb=A\bfx^*$, where $\bfx^*$ is constructed by first generating a random nonnegative vector with sparsity parameter $d=0.6$ and then normalized it so that its components sum to 1.
Thus 0 is the optimal value of this problem. We use the same initial point $\bfx_0 = \bfone_n/n$ for all methods.

In comparison to the experiments in the previous subsection, we introduce four additional algorithms: AFW and PFW, along with their respective versions integrated into the rSFW framework, denoted as rSFW-A and rSFW-P.
Details of these methods are provided in Table~\ref{table:algorithms}.

We briefly explain here how the direction-correction $\bfg$ is computed within the general framework \eqref{Modified-FW} for both the rSFW-A and rSFW-P algorithms.
Based on Alg.~\ref{Alg: rSFW}, during the $k$-th outer loop and the $j$-th inner loop, let $S^{(k,j)}\subset \mathcal{V}(S(\widehat{\bfx}_{k-1},\widehat{d}_{k-1}))$ denote the active set corresponding to the point $\bfp_{j-1}$. Thus, $\bfp_{j-1}$ can be represented as $\bfp_{j-1}=\sum_{\bfv\in S^{(k,j)}}\alpha_{\bfv}\bfv$ where $\alpha_{\bfv}>0$.
Let $\bfv_j={\arg\max}_{\bfv\in S^{(k,j)}}\langle \nabla f(\bfp_{j-1}),\bfv \rangle$.
For the rSFW-A method, the direction-correction is computed as:
\begin{equation}\label{Eq: AFW-correction}
\bfg_j =
\begin{cases}
\frac{1}{1-\alpha_{\bfv_j}}\bfp_{j-1}-\bfy-\frac{\alpha_{\bfv_j}}{1-\alpha_{\bfv_j}}\bfv_j & \text{if } \Delta_j<0, \\
\bfzero & \text{if } \Delta_j\geq 0,
\end{cases}
\end{equation}
where $\Delta_j:=\langle -\nabla f(\bfg_{j-1}),\bfy_j-\bfp_{j-1}\rangle-\langle -\nabla f(\bfg_{j-1}),\bfp_{j-1}-\bfv_j\rangle.$

For the rSFW-P method, the direction-correction is computed as:
\begin{equation}\label{Eq: PFW-correction}
    \bfg_j =\bfp_{j-1}-(1-\alpha_{\bfv_j})\bfy_j-\alpha_{\bfv_j}\bfv_j.
\end{equation}

Finally, we update the point $\bfp_j$ using the iteration:
\begin{equation}\label{Eq: rSFW_P_update_2}
    \bfp_j\gets(1-\delta_j)\bfp_{j-1}+\delta_j(\bfy_j+\bfg_j),
\end{equation}
which replaces the original iteration.

\begin{figure}[ht]
    \centering
    \subfloat{
        \includegraphics[width=0.45\textwidth]{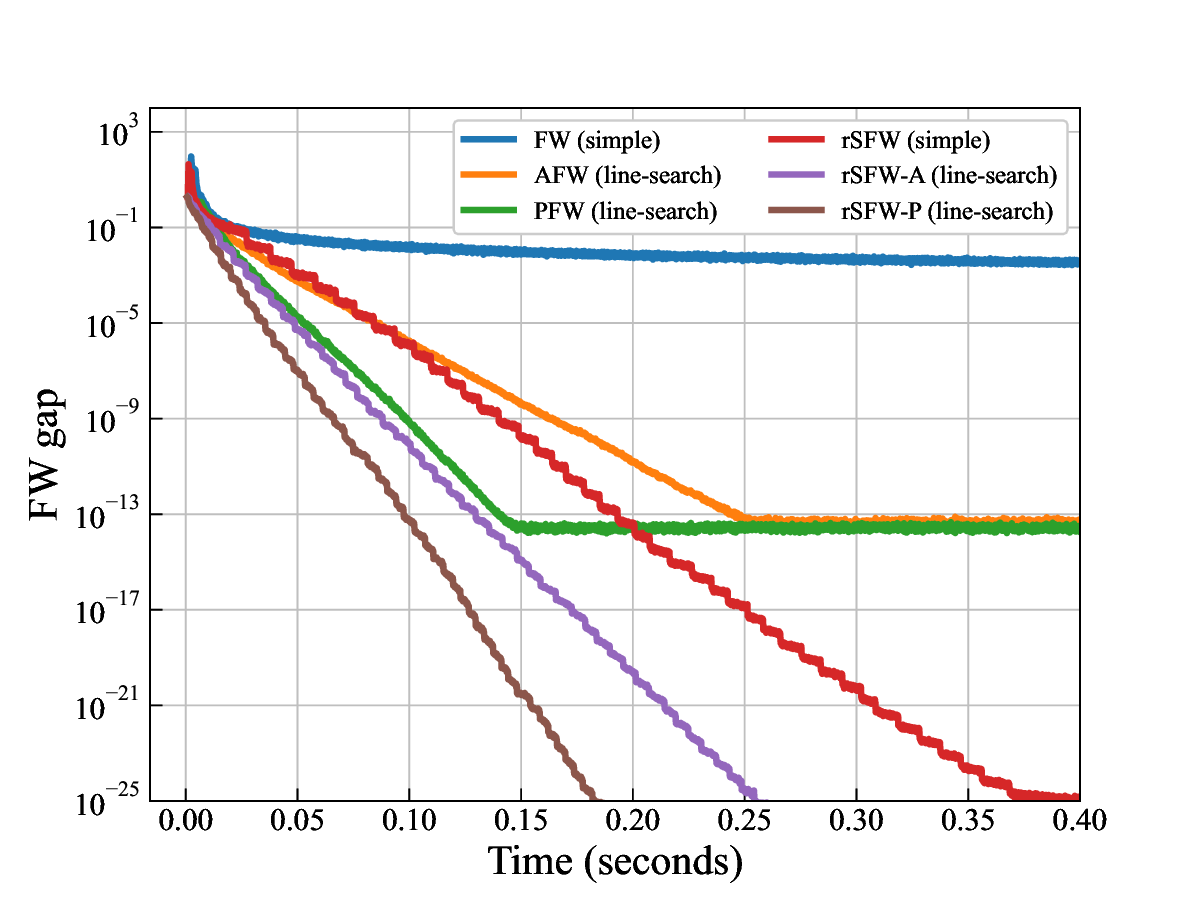}
        \label{Fig: time2gap_Simplex}
    }
    ~
    \subfloat{
        \includegraphics[width=0.45\textwidth]{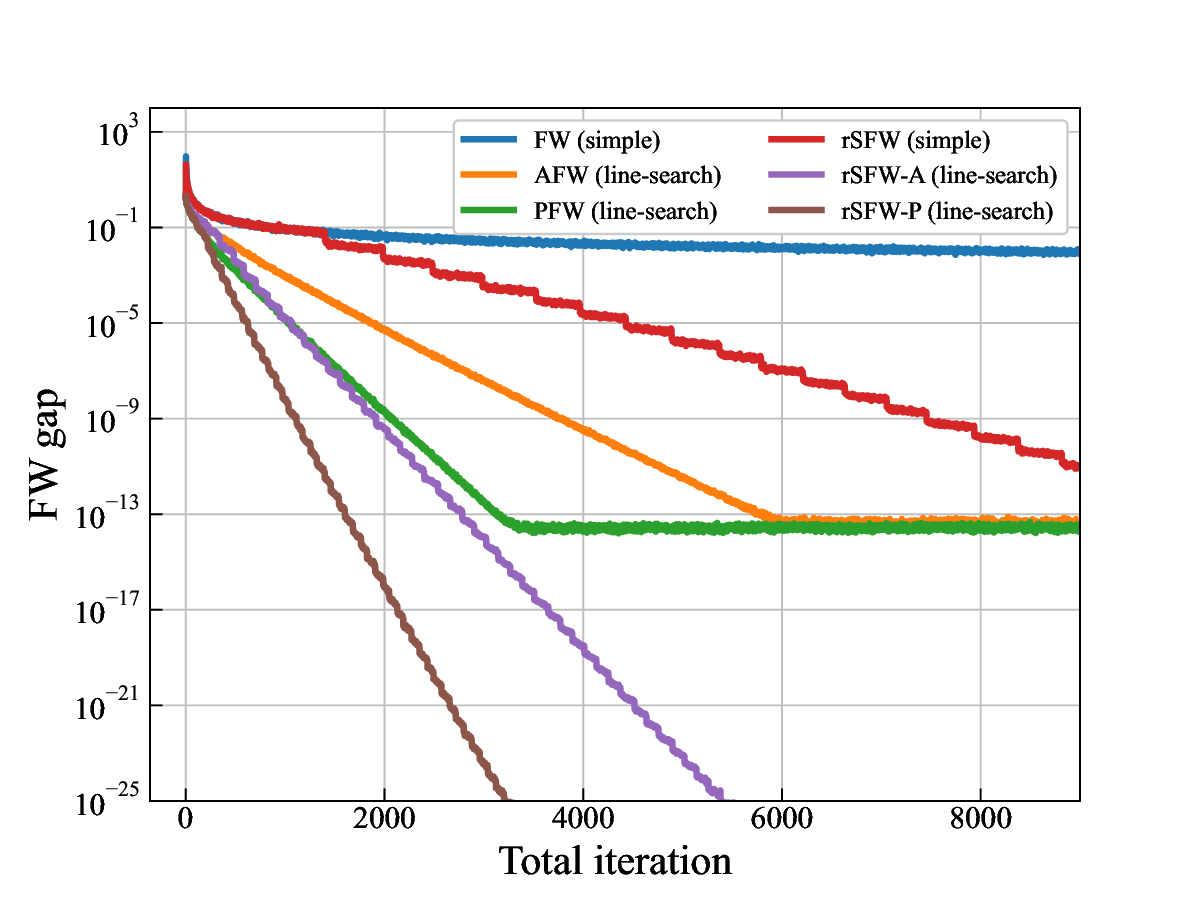}
        \label{Fig: iter2gap_Simplex}
    }
    \caption{FW gap vs time/iterations on the Simplex-constrained least squared  problem with $(m,n)=(800,200)$.}\label{Fig: FW_compare}
\end{figure}

The results are given in Figure \ref{Fig: FW_compare}. rSFW-P demonstrates superior performance compared to all other algorithms, excelling in both the number of iterations and running time, achieving nearly twice the efficiency of PFW.
Additionally, both framework-based acceleration algorithms exhibit substantial performance improvements over the standalone rSFW framework.

\section{Conclusion} \label{Section-Conclusion}

In this paper, we introduced a novel oracle: SLMO, %
which leverages the advantageous geometric properties of the unit simplex.
This design enables SLMO to be implemented with the same computational complexity as the standard linear optimization oracle, preserving the efficiency of the Frank-Wolfe framework.
Building on this oracle, we proposed two new variants of the classical Frank-Wolfe algorithm: the Simplex Frank-Wolfe (SFW) and refined Simplex Frank-Wolfe (rSFW) algorithms. Both methods achieve linear convergence for smooth and strongly convex optimization problems over polytopes.
The linear convergence rates of these methods depend only on the condition number of the objective function, the polytope  s quantity, and the problem  s dimension, demonstrating their scalability and robustness in various settings.

The purpose of this paper is to develop the basic framework for the new SFW methods and demonstrate
that they are highly competitive. 
We do so with the simplest setting of $f$ being strongly convex and smooth.
We 
made no attempt to weaken such assumption except pointing out that the obtained results should also hold 
under the quadratic growth condition.
An immediate question would be to extend the methods to convex or even nonconvex setting.
Furthermore, LLOO proposed in \cite{garber2016linearly} was an elegant framework and it was largely omitted
from recent surveys on FW methods. 
To our best knowledge, SFW was the first LLOO instance that was extensively tested and compared with other 
popular FW methods. 
An intriguing question is whether there exist
alternative LLOO approaches that simultaneously satisfy the following criteria: (i) adhering to the LLOO framework, (ii) enabling fast and accurate computation, and (iii) incurring significantly lower overhead compared to projection-based methods?
We leave those topics to our future research.

\section*{acknowledgement}
    This work was supported by the National Natural Science Foundation of China (Grant No. 12171271) and by Hong Kong RGC General Research Fund PolyU/15303124.

\appendix
\section*{Appendix}
\section{Proof of Lemma \ref{Lemma: Simplex ball}}\label{Section: app_proof_simplexball}

\begin{proof} 
    (1) By the definition of the simplex ball $S(\bfx,d)$, we have
 \begin{align*}
 	S(\bfone_n/n, 1/n) &= \frac 1n \bfone_n + n \times \frac 1n S_0 = \frac 1n \bfone_n + S_0
 	 = S_n .
 \end{align*}
Furthermore, we have
\[
  \Conv\left\{
   n \bfe_i - \bfone_n: \ i \in [n]
  \right\} = n \Conv\left\{
   \bfe_i: \ i \in [n]
  \right\} - \bfone_n = n S_n - \bfone_n = nS_0 .
\]
Consequently, the characterization \eqref{SimplexBall} holds.


    (2)
    On one hand, for every \(\bfy \in S(\bfx, d) \cap S_n\), there exists \(\bflambda \in S_n\) such that \(y_i = x_i - d + nd\lambda_i, \forall i \in [n]\). Define for each \(i \in [n]\),
    \[
    \widehat{\lambda}_i = \left\{
    \begin{array}{ll}
    \frac{nd\lambda_i}{\sum_{j=1}^n \min\{d, x_j\}} & \text{if } x_i \geq d, \\ [2ex]
    \frac{x_i - d + nd\lambda_i}{\sum_{j=1}^n \min\{d, x_j\}} & \text{if } x_i < d.
    \end{array}
    \right .
    \]
    Since \(\bfy \in S_n\) and \(\bflambda \in S_n\), we have \(nd\lambda_i \geq 0\) and \(x_i - d + nd\lambda_i = y_i \geq 0\), which shows that \(\widehat{\lambda}_i \geq 0\) for each \(i \in [n]\). Moreover, let
    \(
    {\mathcal{I}}_- := \{i \in [n] \mid x_i < d\}\) be an index set. Then we have
    \[
    \sum_{i=1}^n \widehat{\lambda}_i = \frac{\sum_{i \in {\mathcal{I}}_-} (x_i - d) + nd}{\sum_{i=1}^n \min\{d, x_i\}} = \frac{\sum_{i \in {\mathcal{I}}_-} x_i + (n - |{\mathcal{I}}_-|) d}{\sum_{i \in {\mathcal{I}}_-} x_i + (n - |{\mathcal{I}}_-|) d} = 1.
    \]
    As above, we have verified that
    $\widehat{\bflambda} \in S_n$.
    We now show that $\bfy = (\widehat{\bfx} - \widehat{d}\bfone_n) + n\widehat{d}\widehat{\bflambda}$, where \(\widehat{\bfx}\) and \(\widehat{d}\) are defined in \eqref{Eq: definition_d}.
    This follows from the fact that, for each \(i \in [n]\),
    \begin{align*}
    & \widehat{x}_i - \widehat{d} + n\widehat{d}\;\widehat{\lambda}_i \\
    =& \max\{x_i, d\} + \widehat{d} - d - \widehat{d} + n \frac{\sum_{j=1}^n \min\{d, x_j\}}{n}
    \times \frac{\min\{x_i - d, 0\} + nd\lambda_i}{\sum_{j=1}^n \min\{d, x_j\}} \\
    =& \max\{x_i, d\} - d + \min\{x_i - d, 0\} + nd\lambda_i 
    =\; x_i - d + nd\lambda_i = y_i.
    \end{align*}
    Thus, we have \(\bfy \in S(\widehat{\bfx}, \widehat{d})\), leading to \(S_n \cap S(\bfx, d) \subset S(\widehat{\bfx}, \widehat{d})\).

    On the other hand, for every \(\bfy \in S(\widehat{\bfx}, \widehat{d})\), there exists \(\widehat{\bflambda} \in S_n\) such that \(y_i = \widehat{x}_i - \widehat{d} + n\widehat{d}\; \widehat{\lambda}_i\). Due to the fact that for \(i \in [n]\), \(y_i \geq \widehat{x}_i - \widehat{d} = \max\{x_i, d\} - d \geq 0\) and
    \[
    \sum_{i=1}^n y_i = \sum_{i=1}^n \max\{x_i, d\} + n\widehat{d} - nd 
    = \sum_{i=1}^n \max\{x_i, d\} + \sum_{i=1}^n \min\{x_i, d\} - nd
    = \sum_{i=1}^n x_i = 1,
    \]
    we have \(\bfy \in S_n\). We now turn to show that \(\bfy \in S(\bfx, d)\).

    Let \(\lambda_i := \frac{\max\{x_i, d\} - x_i + \sum_{j=1}^n \min\{x_j, d\} \widehat{\lambda}_i}{nd}\). It is not difficult to verify that \(\lambda_i \geq 0\) and \(\sum_{i=1}^n \lambda_i = 1\). Moreover, we have
    \begin{align*}
    & x_i - d + nd\lambda_i 
    = x_i - d + \max\{x_i, d\} - x_i + \sum_{j=1}^n \min\{x_j, d\} \widehat{\lambda}_i \\
    =& (\max\{x_i, d\} + \widehat{d} - d) - \widehat{d} + n\widehat{d}\;\widehat{\lambda}_i 
    = \widehat{x}_i - \widehat{d} + n\widehat{d}\;\widehat{\lambda}_i = y_i,
    \end{align*}
   implying \(\bfy \in S(\bfx, d)\). Thus \(S(\widehat{\bfx}, \widehat{d}) \subset S_n \cap S(\bfx, d)\).
    This finishes the proof for \eqref{Eq: definition_d}.

    We now proceed to prove \eqref{Eq: intsect_simplex_balls}.
    Following the definition of $S(\bfx, d)$, we have
    \[
      S(\bfx, d) = (nd) S_n + (\bfx - d\bfone_n).
    \]
    Translating to $(\bfx_1, d_1)$ and $(\bfx_2, d_2)$, we have
    \begin{equation*}
    \begin{aligned}
        S(\bfx_1,d_1)&= (nd_1) S_n +(\bfx_1-d_1\bfone_n), \\
        S(\bfx_2,d_2)&= (nd_2) S_n + (\bfx_2 - d_2 \bfone_n) \\
        &= (nd_1) \left[ n \times \frac{d_2}{nd_1} S_n +
        \left(
          \frac{\bfx_2 - (\bfx_1 - d_1 \bfone_n)}{ nd_1} - \frac{d_2}{nd_1} \bfone_n
        \right)
        \right] +(\bfx_1-d_1\bfone_n )\\
        &=
        nd_1S \left(\frac{\bfx_2-(\bfx_1-d_1\bfone_n)}{nd_1},\frac{d_2}{nd_1} \right)+(\bfx_1-d_1\bfone_n).
    \end{aligned}
    \end{equation*}
    Therefore,
    \[
     S(\bfx_1,d_1) \cap S(\bfx_2,d_2)
     = (\bfx_1-d_1\bfone_n ) + (nd_1) \left[
      S_n \cap \left( \frac{\bfx_2-(\bfx_1-d_1\bfone_n)}{nd_1},\frac{d_2}{nd_1} \right)
     \right]
    \]
    Using \eqref{Eq: definition_d}, we have
    \[
    S_n \cap S \left(\frac{\bfx_2-(\bfx_1 - d_1\bfone_n)}{nd_1},\frac{d_2}{nd_1} \right)=S(\widehat{\bfx},\widehat{d}),\]
    where
    \begin{equation*}
    \begin{aligned}
        \widehat{d} &= \frac{1}{n}\sum_{i=1}^n\min \left\{\frac{d_2}{nd_1},\frac{x_2(i)-x_1(i)+d_1}{nd_1} \right\} 
        = \frac{\sum_{i=1}^n\min\{d_2,x_2(i)-x_1(i)+d_1\}}{n^2d_1} \\
        &\stackrel{(a)}{=}  \frac{1+\sum_{i=1}^n\min\{d_1-x_1(i),d_2-x_2(i) \}}{n^2d_1}\\
        \widehat{x}_i &= \max \left\{\frac{d_2}{nd_1}, \frac{x_2(i)-x_1(i)+d_1}{nd_1} \right\}+ \left(\widehat{d}-\frac{d_2}{nd_1} \right),\quad \forall i\in [n].
    \end{aligned}
    \end{equation*}
    Here, $(a)$ follows from the fact $\bfx_2\in S_n$.
    We then have
 \begin{eqnarray*}
 	S(\bfx_1,d_1) \cap S(\bfx_2,d_2)
 	&=&  (nd_1) S(\widehat{\bfx},\widehat{d}) + (\bfx_1-d_1\bfone_n ) \\
 	&=& (nd_1) \left[
 	 (n\widehat{d}) S_n + (\widehat{\bfx} - \widehat{d} \bfone_n)
 	\right] + (\bfx_1-d_1\bfone_n ) \\
 	&=& n (n d_1 \widehat{d}) S_n + \left[
 	 (\bfx_1 + nd_1 \widehat{\bfx} - d_1 \bfone_n) - (n d_1 \widehat{d}) \bfone_n
 	\right] \\
 	&=& S\Big(
 	\bfx_1 + nd_1 \widehat{\bfx} - d_1 \bfone_n,\  n d_1 \widehat{d}
 	\Big) 
 	= S(\bfx_3, d_3),
 \end{eqnarray*}
  where
    \begin{equation*}
    \begin{aligned}
    	 d_3 &= nd_1\widehat{d}= \frac{1+\sum_{i=1}^n\min\{d_1-x_1(i),d_2-x_2(i) \}}{n},\\
        x_3(i) &= nd_1\widehat{x}(i)+(x_1{(i)}-d_1)\\
        &= \max\{d_2, x_2(i)-x_1(i)+d_1 \}+(nd_1\widehat{d}-d_2)+(x_1{(i)}-d_1) \\
        &=  \max\{d_2, x_2(i)-x_1(i)+d_1 \} + (x_1{(i)}-d_1 -d_2) + d_3 \\
        &= \max\{x_1(i)-d_1,x_2(i)-d_2 \}+ d_3.
    \end{aligned}
    \end{equation*}
    Thus, we have proven \eqref{Eq: intsect_simplex_balls}.

    (3) Since the optimal solution $\bfy^*$ lies on the extreme point of $S(\bfx,d)$, we have
    \begin{equation*}
    \begin{aligned}
        \bfy^* &= {\argmin}_{\bfy=\bfx+d(n\bfe_i-\bfone_n),i\in [n]} \langle \bfc,\bfx+d(n\bfe_i-\bfone_n)\rangle \\
        &= {\argmin}_{\bfy=\bfx+d(n\bfe_i-\bfone_n),i\in [n]} \langle \bfc,e_i\rangle 
        = \bfx+d(n\bfe_{i^*}-\bfone_n),
    \end{aligned}
    \end{equation*}
    where $i^* = \argmin_{i\in [n]}c_i$.

    (4) By \eqref{SimplexBall-New}, for any points $\bfy_1,\bfy_2\in S(\bfx,d)$, there exist $\bflambda_1,\bflambda_2\in S_n$ such that $\bfy_i=(\bfx-d\bfone_n)+nd\bflambda_i$ for $i\in [2]$. Thus, we have
    \begin{equation*}
        \max_{\bfy_1,\bfy_2\in S(\bfx,d)}\lVert \bfy_1-\bfy_2 \rVert = nd\cdot\max_{\bflambda_1,\bflambda_2\in S_n} \lVert \bflambda_1-\bflambda_2\rVert= \sqrt{2}nd.
    \end{equation*}

    (5) Let $\bflambda=\frac{1}{n}(\bfone_n+\frac{\bfy-\bfx}{d})$. Since $\lVert\bfy-\bfx\rVert\leq d$, we have $\lambda_i\geq 0$. Moreover, since $\bfx,\bfy\in S_n$, we have $\sum_{i=1}^n\lambda_i=1$, which implies $\bflambda\in S_n$. Thus by \eqref{SimplexBall-New}, $\bfy=(\bfx-d\bfone_n)+nd\bflambda\in S(\bfx,d)$.
        We now turn to the last part of Lemma \ref{Lemma: Simplex ball}(5). This simply follows from
    \begin{equation*}
        \max_{\bfy\in S(\bfx,d)}\lVert \bfy-\bfx \rVert = nd\cdot\max_{\bflambda\in S_n} \lVert \bflambda-\frac{\bfone_n}{n}\rVert = \sqrt{n(n-1)}d\leq nd.
    \end{equation*}
    The proof is completed.
\end{proof}


\section{Proofs for Section \ref{Section: gen_P}}\label{Section: app_prove_gen_P}

\subsection{A Useful Bound}

The proof of Lemma \ref{Lemma: SLMO_P} relies on the following lemma, whose proof used some key technical results established in \cite{garber2016linearly}.
In particular, for given $\bfx, \bfy \in \calP$ and $\bflambda \in \calM({\bfx})$,
there must exist $\bfz \in \calP$ and $\gamma \in [0, 1]$ such that
\[
	\bfy = \gamma \bfx + (1-\gamma) \bfz 
	= \gamma \sum_{j=1}^N \lambda_j \bfv_j + (1-\gamma) \bfz 
	= \sum_{i=1}^N \Big(
	\lambda_j - \lambda_j (1-\gamma)
	\Big) \bfv_j + (1-\gamma) \bfz .
\]
Let $\Delta_j := \lambda_j (1-\gamma) \in [0, \lambda_j]$. We then have
$
 1- \gamma = \sum_{j=1}^N \Delta_j =: \Delta .
$
To put another way, the point $\bfy$ can always be represented by
\be \label{y-Representation}
  \bfy = \sum_{j=1}^N(\lambda_j-\Delta_j)\bfv_j+ \Delta \bfz,
\ee
for some $\bfz \in \calP$, $\Delta_j \in [0, \lambda_j]$.
Since $\calP$ is compact. There must exist a representation of \eqref{y-Representation} with the smallest $\Delta$
among all such representations.
An important fact established in \cite[lemma~5.3]{garber2016linearly} is that
the minimal value $\Delta$ can be bounded.
We refine this bound below for the largest $\Delta_j$ in $\Delta$.

\begin{lemma}\label{Lemma: Delta_bound}
    Let $\bfx, \bfy \in\calP$ with $\bflambda \in \calM(\bfx)$
    Let $\bfy$ be represented as in \eqref{y-Representation} with
    $\Delta$ having been minimized. Then it holds that
    \begin{equation*}
        \max_{i\in [N]}\{\Delta_i\}\leq \frac{\psi}{\xi}\lVert \bfx-\bfy\rVert.
    \end{equation*}
\end{lemma}
\begin{proof}
    The claim is trivial for the case $\sum_{i=1}^N\Delta_i=0$.
    Now we suppose that $\sum_{i=1}^N\Delta_i>0$ (i.e., at least one $\Delta_i >0$).
    The following index sets $C(\bfz)$ and $C_0(\bfz)$ are defined in \cite{garber2016linearly}.
    We simply describe them and use some established results relating to them.
    Denote the index set $C(\bfz) :=\{j\in [m]\mid A_2(j)\bfz=b_2(j)\}$.
    By \cite[Lemma 5.3]{garber2016linearly} we have $C(\bfz)\neq\emptyset$
    since one $\Delta_i>0$.
    Let $C_0(\bfz)\subseteq C(\bfz)$ be such that the set $\{A_2(j)\}_{j\in C_0(\bfz)}$ forms a basis for the set $\{A_2(j)\}_{j\in C(\bfz)}$. Denote by $A_{2,{z}}\in\mathbb{R}^{|C_0(\bfz)|\times n}$ consisting of the set $\{A_2(j)\}_{j\in C_0(\bfz)}$. By definition we have $\|A_{2,{z}}\|\leq \psi$. Then we obtain
    \begin{equation*}
    \begin{aligned}
        \|\bfx-\bfy \|^2
        &=\left\|\sum_{i\in [N]:\Delta_i>0}\Delta_i(\bfv_i-\bfz) \right\|^2 \\
        & \ge \frac{1}{\psi^2}\sum_{j\in C_0(\bfz)}\left (\sum_{i\in [N]:\Delta_i>0}\Delta_i(b_2(j)-A_2(j)\bfv_i) \right )^2 \\
        &\stackrel{(a)}{\geq} \frac{1}{\psi^2}\sum_{j\in C_0(\bfz)} \sum_{i\in [N]:\Delta_i>0}\Delta_i^2(b_2(j)-A_2(j)\bfv_i)^2 \\
        &= \frac{1}{\psi^2}\sum_{i\in [N]:\Delta_i>0}\sum_{j\in C_0(\bfz)} \Delta_i^2(b_2(j)-A_2(j)\bfv_i)^2,
    \end{aligned}
    \end{equation*}
    where the first inequality is established in the proof of \cite[Lemma~5.5]{garber2016linearly},
     $(a)$ follows from the fact that for any $i\in [N]$, and any $j\in C_0(\bfz)$ we have $b_2(j)-A_2(j)\bfv_i\geq 0$. Combining \cite[Lemma 5.3]{garber2016linearly} and \cite[Lemma 5.4]{garber2016linearly}, we obtain that for all $i\in [N]$ such that $\Delta_i>0$ there exists $j\in C_0(\bfz)$ such that $b_2(j)-A_2(j)\bfv\geq \xi$. Hence,
    \begin{equation*}
        \|\bfx-\bfy\|^2\geq \frac{\xi^2}{\psi^2}\sum_{i\in [N]:\Delta_i>0}\Delta_i^2\geq \frac{\xi^2}{\psi^2}\max_{i\in [N]}\{\Delta_i^2\}.
    \end{equation*}
    Thus we conclude that $\max_{i\in [N]}\{\Delta_i\}\leq \frac{\psi}{\xi}\lVert \bfx-\bfy\rVert$.
\end{proof}

\subsection{Proof of Lemma \ref{Lemma: SLMO_P}}

\begin{proof}
    We begin by proving the first part. Write $\bfx=\sum_{i=1}^N\lambda_i\bfv_i$ for $\bflambda\in S_N$ and express $\bfy=\sum_{i=1}^N(\lambda_i-\Delta_i)\bfv_i+(\sum_{i=1}^N\Delta_i)\bfz$, where $\Delta_i\in [0,\lambda_i],\forall i\in [N]$ and $\bfz\in\calP$. Here, the sum $\Delta = \sum_{i=1}^N\Delta_i$ is minimized (as in Lemma \ref{Lemma: Delta_bound}). We then have
    \begin{equation*}
        \max_{i\in [N]}\{\Delta_i\}
         \le \frac{\psi}{\xi}\lVert \bfx-\bfy\rVert
         \le \frac{\psi}{\xi} \times \frac{d D}{\eta}
        = d,
    \end{equation*}
   where the first inequality used Lemma \ref{Lemma: Delta_bound}, the second inequality used the assumption
   $\| \bfx - \bfy \| \le (dD)/\eta$, and the last equation is by the definition of
   $\eta$ in \eqref{eq: eta_def}.
    Express $\bfz$ as $\bfz=\sum_{i=1}^N\lambda_i'\bfv_i$, where $\bflambda'\in S_N$. We can then rewrite $\bfy$ as follows:
     \[
    		\bfy
    		=\sum_{i=1}^N(\lambda_i-\Delta_i+\Delta\lambda_i')\bfv_i 
    		= \sum_{i=1}^N\left((\lambda_i-d)+d-\Delta_i+\Delta\lambda_i'\right)\bfv_i.
    	\]
    Since $\max_{i\in [N]}\{\Delta_i\}\leq d$, we have $d-\Delta_i+\Delta\lambda_i'\geq 0$ for all $i\in [N]$. Moreover, the sum $\sum_{i=1}^N(d-\Delta_i+\Delta\lambda_i)=Nd$, which implies that $\frac{(d-\Delta_i+\Delta\lambda_i)_i}{Nd}\in S_N$. By the definition in \eqref{SimplexBall-New}, we have $\bflambda_y:=(\lambda_i-\Delta_i+\Delta\lambda_i')\in S(\bflambda,d)$, thus $\bfy\in S_{\calP}(\bfx,d)$. Moreover, we have
    \[
    \langle \bfy,\bfc \rangle = \langle \bflambda_y,\bfc_{ext} \rangle \geq \langle \bflambda^*,\bfc_{ext} \rangle = \langle \bfy^*,\bfc \rangle.
    \]

    We now turn to prove the second part. Referring to Algorithm~\ref{Alg: SLMO_P}, we note that
    \[
      \bflambda_+ - d\bfone_N = \max\{ \bflambda_x, d\bfone_N\} - d \bfone_N = \bflambda_x - \min\{\bflambda_x, d \bfone_N\}
    \]
    and
    \[
      N \widehat{d} = \sum_{i=1}^N \min\{\lambda_x(i), d \}.
    \]
    Denote
    ${\mathcal I}_+(\bflambda_x) :=\{i\in [N]\mid \lambda_x(i)>0\}$,
    $\delta_i := \min\{\lambda_x(i),d\}$, and
    $\delta := \sum_{i\in {\mathcal I}_+(\bflambda_x)}\delta_i$, the optimal solution
    $\bfy^*$ produced by Algorithm~\ref{Alg: SLMO_P} has
    \begin{equation*}
        \bfy^*=  \bflambda_+ + d \bfone_N + N \widehat{d} \bfe_{i^*}
        = \sum_{i\in {\mathcal I}_+(\bflambda_x)}(\lambda_x(i)-\delta_i)\bfv_i+\delta\bfv_{i^*},
    \end{equation*}
Thus we have that
    \begin{equation*}
    \begin{aligned}
        \lVert \bfx-\bfy^*\rVert &= \lVert \sum_{i\in {\mathcal I}_+(\bflambda_x)}\min\{\lambda_x(i),d\}(\bfv_i-\bfv_{i^*})\rVert \\
        &\leq \sum_{i\in {\mathcal I}_+(\bflambda_x)}\min\{\lambda_x(i),d\}\lVert \bfv_i-\bfv_{i^*}\rVert 
        \leq |{\mathcal I}_+(\bflambda_x)|dD\leq (n+1)dD.
    \end{aligned}
    \end{equation*}
\end{proof}

\subsection{Proof of Theorem \ref{Thm: Convergence_sFW_P}}


\begin{proof} 
The proof follows the framework of the proof of Theorem \ref{Thm: Convergence-SL1} and Lemma \ref{Lemma: SLMO_P}.
We first claim that $\bfx^*\in S_{\calP}(\bfx_k,\frac{\eta}{D}d_k)$ and that $f(\bfx_k)-B_k\leq \frac{\mu d_k^2}{2}$. We prove this by induction.
    First, we have
    \begin{equation*}
        \frac{\mu d_0^2}{2}=f(\bfx_0)-B_0\geq f(\bfx_0)-f^*\stackrel{(a)}{\geq} \frac{\mu}{2}\lVert \bfx_{0}-\bfx^* \rVert^2,
    \end{equation*}
    where $(a)$ comes from \eqref{Eq: strongly_convex_property}.
    This implies that $\lVert \bfx_{0}-\bfx^* \rVert\leq d_0$, and by Lemma \ref{Lemma: SLMO_P}, we have $\bfx^*\in S_{\calP}(\bfx_0,\frac{\eta}{D}d_0)$.
    Therefore, the claim holds for $k=0$.

    Now suppose that $\bfx^*\in S_{\calP}(\bfx_t,\frac{\eta}{D}d_t)$ and $f(\bfx_t)-B_t\leq \frac{\mu d_t^2}{2}$ for all $t\leq k-1$.
    Let $\gamma :=\frac{\mu}{2L(n+1)^2\eta^2}$.
    In the same manner as the proof in Theorem \ref{Thm: Convergence-SL1}, for step size policy \eqref{Eq: step size_line}, \eqref{Eq: step size_smooth} or \eqref{Eq: step size_constant3}, we all have
    \begin{equation*}
    \begin{aligned}
        f(\bfx_k)-B_k\leq & (1-\gamma)(f(\bfx_{k-1})-B_{k-1})+\frac{L\gamma^2}{2}\lVert \bfy_k-\bfx_{k-1} \rVert^2 \\
        \stackrel{(d)}{\leq} & (1-\gamma)\frac{\mu}{2}d_{k-1}^2+\frac{L\gamma^2}{2}(n+1)^2\eta^2d_{k-1}^2 \\
        = & \left[ (1-\gamma)\frac{\mu}{2} + \frac{L\gamma^2(n+1)^2\eta^2}{2} \right]d_{k-1}^2,
    \end{aligned}
    \end{equation*}
    where $(d)$ is due to our inductive hypothesis and Lemma~\ref{Lemma: SLMO_P}.
    By plugging in the value of $\gamma$, and using $1-x\leq e^{-x}$, we have that
    \begin{equation*}
        f(\bfx_k)-B_k\leq \frac{\mu}{2}(1-\frac{\mu}{4L(n+1)^2\eta^2})d_{k-1}^2\leq \frac{\mu}{2}e^{-\frac{\mu}{4L(n+1)^2\eta^2}}d_{k-1}^2.
    \end{equation*}
    Combining the above inequality with the fact that $f(\bfx_k)-B_k=\frac{\mu}{2}(\sqrt{\frac{2(f(\bfx_k)-B_k)}{\mu}})^2$, and by the definition of $d_k$, we conclude that $f(\bfx_k)-B_k\leq \frac{\mu d_k^2}{2}$.
    By the inductive hypothesis, we know that $\bfx^*\in S_{\calP}(\bfx_t,d_t)$ holds for all $t\leq k-1$. Thus $B_{t+1}^w$ is a valid lower bound of $f^*$, and consequently, $B_k$ is also a lower bound of $f^*$. Now by \eqref{Eq: strongly_convex_property}, we have
    \begin{equation}\label{Eq: extreme_S_p}
\lVert \bfx_k-\bfx^*\rVert^2\leq (2/\mu)(f(\bfx_k)-f^*)\leq (2/\mu)(f(\bfx_k)-B_k)\leq d_k^2.
    \end{equation}
    This implies that $\bfx^*\in S_{\calP}(\bfx_k,\frac{\eta}{D}d_k)$ by Lemma \ref{Lemma: SLMO_P}. Therefore, we have completed the proof of the claim.

    We now start to prove the conclusion in Theorem \ref{Thm: Convergence_sFW_P}.
    From the earlier proof, we know that $B_k\leq f^*$, thus confirming the first part of the inequality.
    By the definition of $d_k$ and the established claim, we have
    $$
    f(\bfx_k)-B_k\leq \frac{1}{2}\mu d_k^2 \leq \frac{\mu d_0^2}{2}e^{-\frac{\mu}{4L(n+1)^2\eta^2}k} .
    $$
    The proof is thus completed.
\end{proof}

\subsection{Proof of Theorem \ref{Thm: Convergence_Sp_sFW_P}}

The proof of Theorem \ref{Thm: Convergence_Sp_sFW_P} relies on the following lemma.
\begin{lemma}\label{Lemma: Simplex_ball_P}
    For any $\bflambda_1,\bflambda_2 \in S(\bflambda_x,d)$, define the two corresponding points:
    \[
     \bfy_j=\sum_{i=1}^N\lambda_j (i)\bfv_i  \in S_{\calP} ( \bfx, d), \quad j=1, 2.
    \]
    We must have
    $
      \| \bfy_1 - \bfy_2 \| \le (n+1) d D.
    $
\end{lemma}

\begin{proof}
We note that $S_{\calP} ( \bfx, d)$ is compact. Let $\mathcal{V}(S)$ denote the set of its vertices. Obviously,
we have
\begin{equation*}
	\|\bfy_1-\bfy_2\|\leq \max_{\bfu_1,\bfu_2\in\mathcal{V}(S)}\|\bfu_1-\bfu_2\|.
\end{equation*}
For $\bfu_1 \in \calV(S)$ being a vertex of $S_{\calP} ( \bfx, d)$,
according to the separation theorem \cite{rockafellar1997convex} that there exists a vector $\bfc_1 \in \mathbb{R}^N$
such that
\[
  \langle \bfc_1, \; \bfu_1 \rangle < \langle \bfc_1, \; \bfu \rangle \quad \mbox{for all} \
 \bfu \in \mathcal{V} (S) \setminus \{ \bfu_1 \} .
\]
In other words, $\bfu_1$ is the unique solution of the following problem:
\[
  \min \; \langle \bfc_1, \; \bfu \rangle \quad \mbox{s.t.} \quad \bfu \in S_{\calP} (\bfx, d) \cap \calP .
\]
It follows Lemma~\ref{Lemma-SLMOP} that $\bfu_1$ can be represented as
\[
 \bfu_1 = \sum_{j=1}^N ( \lambda_x (j) - \delta_j) \bfv_j + \delta \bfz_1 \quad \mbox{for some} \ \bfz_1 \in \calP ,
\]
where $\delta_j := \min\{\lambda_x(j),d\}$ and $\delta :=\sum_{j=1}^N \delta_j$ independent of $\bfz_1$.
Similarly, $\bfu_2$ has a representation:
\[
\bfu_2 = \sum_{j=1}^N ( \lambda_x (j) - \delta_j) \bfv_j + \delta \bfz_2 \quad \mbox{for some} \ \bfz_2 \in \calP .
\]
Therefore, we have
    \begin{eqnarray*}
        \|\bfy_1-\bfy_2\| &\leq & \max_{\bfu_1,\bfu_2\in\mathcal{V}(S)}\|\bfu_1-\bfu_2\| \\
        &\leq& \delta \max_{\bfz_1,\bfz_2\in \calP } \|\bfz_1 - \bfz_2\|
                    \leq \delta D \\
        &=& D \sum_{i \in \mathcal{I}_+(\bflambda_x)} \delta_i  
        \leq  D | \mathcal{I}_+(\bflambda_x) |d
        \leq (n+1)d D,
    \end{eqnarray*}
where $\mathcal{I}_+(\bflambda_x) : = \left\{ i \ | \
 \lambda_i(x) > 0
\right\}$ and we have used $|\mathcal{I}_+(\bflambda_x) | \le (n+1)$.
\end{proof}

\begin{proof}[Proof of Theorem \ref{Thm: Convergence_Sp_sFW_P}]
    We first claim that $\bfx^*\in S_{\calP}({\bfx}_k,{d}_k)$ for any $k\geq 0$ and prove this by induction.
    From the proof of Theorem \ref{Thm: Convergence_sFW_P}, this is true for $k=0$.
    Now suppose that $\bfx^*\in S_{\calP}({\bfx}_{k-1},{d}_{k-1})$ for some $k\geq 1$.
    Note that the inner loop of Alg.~\ref{Alg: rSFW-P} corresponds to the standard Frank-Wolfe algorithm.
    By Theorem \ref{Thm: FW} and Lemma \ref{Lemma: Simplex_ball_P}, we have
    \begin{equation*}
        f(\bfp_j)-f^*\leq \frac{2L}{j+1}\left((n+1)^2{d}_{k-1}D\right)^2 = \frac{2L(n+1)^2{d}_{k-1}^2D^2}{j+1}
    \end{equation*}
    hold for all $j\in [J]$. In the case where the inner loop terminates at $j = J$, we obtain
    \begin{equation*}
        f(\bfx_k)-f^*=f(\bfp_J)-f^*\leq f(\bfp_J)-C_J\leq
        \frac{\mu d_{k}^2D^2}{2\eta^2}.
    \end{equation*}
    Similarly, if the inner loop is interrupted due to lines 9-11 of the algorithm, we still have $f(\bfx_k)-f^*\leq f(\bfp_j)-C_j\leq \frac{\mu}{2\rho^2\eta^2}{d}_{k-1}^2D^2 = \frac{\mu d_{k}^2D^2}{2\eta^2}$.
    Using the fact that $f(\bfx_k)-f^*\geq \frac{\mu}{2}\lVert \bfx_k-\bfx^*\rVert^2$, we have $\lVert \bfx_k-\bfx^*\rVert^2\leq \frac{d_{k}^2D^2}{\eta^2}$, which implies via Lemma \ref{Lemma: SLMO_P} that $\bfx^*\in S_{\calP}({\bfx}_k,{d}_k)$.

    We now start to prove the conclusion in Theorem \ref{Thm: Convergence_Sp_sFW_P}. Since $d_0=\frac{\eta}{D}\sqrt{\frac{2(f(\bfx_0)-B_0)}{\mu}}$ and $d_k=\frac{{d}_{k-1}}{\rho}$, we have
    \begin{equation*}
        f(\bfx_k)-f^*\leq f(\bfx_k)-B_k \leq \frac{\mu d_{k}^2D^2}{2\eta^2} \leq (f(\bfx_0)-B_0)\rho^{-2k}.
    \end{equation*}
\end{proof}

\section{Properties of Some Common Polytopes}\label{Section: app_prop_poly}
\subsection{Carath  odory Representation Examples}\label{Section: app_Cara_rep} $\ $

$\ $\textbf{Hypercube:}
When $\calP$ is a hypercube $B_n:=\{\bfx\in\mathbb{R}^n\mid x_i\in[0,1],\forall i\in [n]\}$, any point $\bfx\in B_n$ can be naturally represented as
\[\bfx=\sum_{i=1}^{n-1}(x_{j_i}-x_{j_{i+1}})\bfv_i+x_{j_n}\bfone_n+(1-x_{j_1})\bfzero_n,\]
where $j_1,\dots,j_n$ is a permutation over $[n]$ such that $x_{j_1}\geq\dots\geq x_{j_n}$ and $\bfv_i$ is a vector with components from $j_1$ to $j_i$ equal to $1$ and the rest equal to $0$.

\textbf{$\ell_1$-ball:}
When $\calP$ is a $\ell_1$-ball $L_n:=\{\bfx\in\mathbb{R}^n\mid \sum_{i=1}^n|x_i|\leq 1\}$, any point $\bfx\in L_n$ can be naturally represented as
\begin{equation*}
    \bfx=\sum_{i=1}^{n-1}|x_i|(\text{sgn}(x_i)\bfe_i)+\left(|x_n|+s_x\right)(\text{sgn}(x_n)\bfe_n)+s_x(-\text{sgn}(x_n)\bfe_n) ,
\end{equation*}
where $s_x={1-\sum_{i=1}^n|x_i|}/{2}$ and $\text{sgn}(x)=1$ if $x\ge 0$ and $-1$ otherwise.

\textbf{Flow polytope:}
Let $G$ be a \textit{directed acyclic graph} (DAG) with a set of vertices $V$ and  edges $E$ such that $|E|=n$.
Let $s,t$ be two vertices in $V$, referred to as the \textit{source} and \textit{target}, respectively.
The $s$ - $t$ flow polytope, here denoted by $\mathcal{F}_{s,t}$, is the set of all unit $s$ - $t$ flows in $G$.
For any point $\bfx\in\mathcal{F}_{s,t}$ and $i\in [n]$, the entry $x_i$ represents the amount of flow through edge $i\in [n]$, where the flow vector $\bfx$ satisfies the flow conservation constraints at each vertex, ensuring that the flow entering any vertex (except $s$ and $t$) equals the flow leaving it.
The extreme points of $\mathcal{F}_{s,t}$ are the extreme unit flows.
To find the Carath  odory representation of a given flow $\bfx\in\mathcal{F}_{s,t}$, we can proceed recursively as follows.

Starting with the flow $\bfx$, we repeatedly perform the following steps until $\bfx=\bfzero_n$:
\begin{enumerate}
    \item Remove all edges with zero flow from the graph.
    \item Identify the edge $i$ corresponding to the smallest non-zero flow in $\bfx$, i.e., $i\gets {\argmin}_{x_i>0}x_i$.
    \item Find the extreme unit flow $\bfv$ in the reduced graph that includes edge $i$.
    \item Subtract $x_i\bfv$ from the current flow, i.e., $\bfx\gets \bfx-x_i\bfv$.
\end{enumerate}
Since each operation eliminates at least one non-zero entry in the current flow, the loop will terminate within at most $m$ steps.
As a result, we obtain a Carath  odory representation of $\bfx$.
This algorithm can be implemented in $O(n^2)$ time when the graph is represented using sparsely structured adjacency matrices.

\subsection{Quantities of Some Common Polytopes}\label{Section: app_quantity}
\quad \textbf{Hypercube:} The diameter of $B_n$ is given by
\[
D(B_n) = \max_{\bfx,\bfy\in B_n}\|\bfx-\bfy \| = \|\bfone_n-\bfzero_n \| = \sqrt{n}.
\]
Since $B_n$ can be represented as
\[
B_n=\left\{\bfx\in\mathbb{R}^n\mid \left(\begin{array}{c}
    I_n  \\
     -I_n
\end{array} \right)\bfx\leq \left(\begin{array}{c}
    \bfone_n \\
    \bfzero_n
\end{array} \right)\right\},
\]
it follows from the definition in Subsection \ref{Subsection: Quantity_Polytope} that $\xi(B_n)=1$ and $\psi(B_n)=1$. Thus, the quantity of $B_n$ is
\[
    \eta(B_n) = {\psi(B_n)D(B_n)}/{\xi(B_n)} = \sqrt{n}.
\]

\textbf{$\ell_1$-ball:}
The diameter of $L_n$ is given by
\[
D(L_n) = \max_{\bfx,\bfy\in L_n}\|\bfx-\bfy \| = \|\bfe_1-(-\bfe_1) \| = 2.
\]
Note that $L_n$ can be described by the linear inequalities system $L_n=\{\bfx\in\mathbb{R}^n\mid A_2\bfx\leq \frac{1}{\sqrt{n}}\bfone_{2^n}\}$, where $A_2\in\mathbb{R}^{2^n\times n}$ is a matrix whose entries are either $\pm \frac{1}{\sqrt{n}}$ and whose rows all have unit $\ell_2$ norm.
Following the definition in Subsection \ref{Subsection: Quantity_Polytope}, we have $\xi(L_n)=\frac{2}{\sqrt{n}}$ and
\[
    \frac{1}{\sqrt{n}}= \max_{M\in\mathbb{A}(L_n)}\frac{1}{\sqrt{n}}\|M\|_F \leq \psi(L_n) = \max_{M\in\mathbb{A}(L_n)}\|M\|\leq \max_{M\in\mathbb{A}(L_n)}\|M\|_F = \sqrt{n},
\]
where $\|\cdot\|_F$ denotes the Frobenius norm. Thus, the quantity of $L_n$ can be estimated as
\[
    \eta(L_n) = {\psi(L_n)D(L_n)}/{\xi(L_n)} \in [1,n].
\]

\textbf{Flow Polytope: }
For every two extreme unit flows $\bfx_1,\bfx_2\in\calV(\mathcal{F}_{s,t})$, since $\calV(\mathcal{F}_{s,t})\subseteq \{0,1\}^n$, we have $\|\bfx_1-\bfx_2\|\leq \sqrt{n}$. Thus, the diameter of $\mathcal{F}_{s,t}$ can be estimated as
\[
    D(\mathcal{F}_{s,t}) = \max_{\bfx_1,\bfx_2\in\calV(\mathcal{F}_{s,t})}\|\bfx_1-\bfx_2\| \leq \sqrt{n}.
\]
When representing $\mathcal{F}_{s,t}$ using a system of linear equations and inequalities, the inequality constraints are given by $-I_n\bfx\leq \bfzero_n$.
Thus, by definition, we have $\xi(\mathcal{F}_{s,t})=\psi(\mathcal{F}_{s,t})=1$, leading to
\[
    \eta(\mathcal{F}_{s,t}) = D(\mathcal{F}_{s,t})\leq \sqrt{n}.
\]
It is worth noting that in the numerical experiment in Subsection~\ref{Subsection: num_sFW}, we set $\eta = D =\sqrt{66}$ while the dimension is $n=660$. This choice stems from the specific characteristics of the dataset used in \cite{lacoste2015global,garber2021frank}.
Specifically, we observe that each extreme unit flow contains exactly $33$ entries of $1$, with the remaining entries being $0$.
Consequently, we can estimate $\eta=D\leq \sqrt{66}$.

\section{Backtracking Details}\label{app-backtracking}
The routine of estimating local parameters $L,\mu$ and step-size $\delta$ is shown as follows. For rSFW, if we use the simple step-size, we can employ only the estimates of $L$ and $\mu$ without applying the corresponding step size. 
\begin{algorithm}[!ht]
\footnotesize
	\renewcommand{\algorithmicrequire}{\textbf{Input:}}
	\renewcommand{\algorithmicensure}{\textbf{Output:}}
	\caption{$\mbox{Backtracking-Routine}(\bfx,\bfd,L,\mu,\delta_{\text{max}})$}
	\label{Alg: backtracking}
	\begin{algorithmic}[1]
        \REQUIRE Iterate $\bfx\in\calP$, update direction $\bfd$, previous estimation $L$ and $\mu$.
        \STATE Choose $\tau_1>1,\tau_2\leq1$.
        \STATE $L\gets\tau_2L,\ \mu\gets\mu/\tau_2$
        \STATE $\delta\gets\min\left\{\frac{\langle\nabla f(\bfx),\bfd \rangle}{L\lVert\bfd\rVert^2},\delta_{\text{max}}\right\}$
        \WHILE{$f(\bfx+\delta\bfd)>f(\bfx)+\delta\langle\nabla f(\bfx),\bfd \rangle+\frac{\delta^2L}{2}\lVert\bfd\rVert^2$}
            \STATE $L\gets\tau_1L$
            \STATE $\mu\gets\min\left\{\frac{2(f(\bfx+\delta\bfd)-f(\bfx)-\delta\langle\nabla f(\bfx),\bfd \rangle)}{\delta^2\lVert\bfd\rVert^2},\mu\right\}$
            \STATE $\delta\gets\min\left\{\frac{\langle\nabla f(\bfx),\bfd \rangle}{L\lVert\bfd\rVert^2},\delta_{\text{max}}\right\}$
        \ENDWHILE
        \ENSURE $\delta,L,\mu$.
	\end{algorithmic}
\end{algorithm}



%
%


\begin{thebibliography}{10}
\providecommand{\url}[1]{{#1}}
\providecommand{\urlprefix}{URL }
\expandafter\ifx\csname urlstyle\endcsname\relax
  \providecommand{\doi}[1]{DOI~\discretionary{}{}{}#1}\else
  \providecommand{\doi}{DOI~\discretionary{}{}{}\begingroup
  \urlstyle{rm}\Url}\fi

\bibitem{beck2004conditional}
Beck, A., Teboulle, M.: A conditional gradient method with linear rate of
  convergence for solving convex linear systems.
\newblock Mathematical Methods of Operations Research \textbf{59}, 235--247
  (2004)

\bibitem{gartner2023optimization}
Bernd, G., Jaggi, M.: Optimization for machine learning.
\newblock Lecture Notes CS-439, ETH, Spring 2023  (2023)

\bibitem{bomze2021frank}
Bomze, I.M., Rinaldi, F., Zeffiro, D.: Frank--wolfe and friends: a journey into
  projection-free first-order optimization methods.
\newblock 4OR \textbf{19}(3), 313--345 (2021)

\bibitem{braun2022conditional}
Braun, G., Carderera, A., Combettes, C.W., Hassani, H., Karbasi, A., Mokhtari,
  A., Pokutta, S.: Conditional gradient methods.
\newblock arXiv preprint arXiv:2211.14103  (2022)

\bibitem{Chandrasekaran2012}
Chandrasekaran, V., Recht, B., Parrilo, P.A., Willsky, A.S.: The convex
  geometry of linear inverse problems.
\newblock Foundations of Computational mathematics \textbf{12}(6), 805--849
  (2012)

\bibitem{clarkson2010coresets}
Clarkson, K.L.: Coresets, sparse greedy approximation, and the frank-wolfe
  algorithm.
\newblock ACM Transactions on Algorithms (TALG) \textbf{6}(4), 1--30 (2010)

\bibitem{combettes2021complexity}
Combettes, C.W., Pokutta, S.: Complexity of linear minimization and projection
  on some sets.
\newblock Operations Research Letters \textbf{49}(4), 565--571 (2021)

\bibitem{condat2016fast}
Condat, L.: Fast projection onto the simplex and the l 1 ball.
\newblock Mathematical Programming \textbf{158}(1), 575--585 (2016)

\bibitem{damla2008linear}
Damla~Ahipasaoglu, S., Sun, P., Todd, M.J.: Linear convergence of a modified
  frank--wolfe algorithm for computing minimum-volume enclosing ellipsoids.
\newblock Optimisation Methods and Software \textbf{23}(1), 5--19 (2008)

\bibitem{frank1956algorithm}
Frank, M., Wolfe, P., et~al.: An algorithm for quadratic programming.
\newblock Naval research logistics quarterly \textbf{3}(1-2), 95--110 (1956)

\bibitem{freund2016new}
Freund, R.M., Grigas, P.: New analysis and results for the frank--wolfe method.
\newblock Mathematical Programming \textbf{155}(1), 199--230 (2016)

\bibitem{garber2016faster}
Garber, D.: Faster projection-free convex optimization over the spectrahedron.
\newblock Advances in Neural Information Processing Systems \textbf{29} (2016)

\bibitem{garber2013playing}
Garber, D., Hazan, E.: Playing non-linear games with linear oracles.
\newblock In: 2013 IEEE 54th annual symposium on foundations of computer
  science, pp. 420--428. IEEE (2013)

\bibitem{garber2016linearly}
Garber, D., Hazan, E.: A linearly convergent variant of the conditional
  gradient algorithm under strong convexity, with applications to online and
  stochastic optimization.
\newblock SIAM Journal on Optimization \textbf{26}(3), 1493--1528 (2016)

\bibitem{garber2021frank}
Garber, D., Wolf, N.: Frank-wolfe with a nearest extreme point oracle.
\newblock In: Conference on Learning Theory, pp. 2103--2132. PMLR (2021)

\bibitem{guelat1986some}
Gu{\'e}lat, J., Marcotte, P.: Some comments on wolfe's ‘away step’.
\newblock Mathematical Programming \textbf{35}(1), 110--119 (1986)

\bibitem{guyon2008feature}
Guyon, I., Gunn, S., Nikravesh, M., Zadeh, L.A.: Feature extraction:
  foundations and applications, vol. 207.
\newblock Springer (2008)

\bibitem{hazan2008sparse}
Hazan, E.: Sparse approximate solutions to semidefinite programs.
\newblock In: Latin American symposium on theoretical informatics, pp.
  306--316. Springer (2008)

\bibitem{jaggi2013revisiting}
Jaggi, M.: Revisiting frank-wolfe: Projection-free sparse convex optimization.
\newblock In: International conference on machine learning, pp. 427--435. PMLR
  (2013)

\bibitem{jaggi2010simple}
Jaggi, M., Sulovsk, M., et~al.: A simple algorithm for nuclear norm regularized
  problems.
\newblock In: Proceedings of the 27th international conference on machine
  learning (ICML-10), pp. 471--478 (2010)

\bibitem{joulin2014colocalization}
Joulin, A., Tang, K., Fei-Fei, L.: Efficient image and video co-localization
  with frank-wolfe algorithm.
\newblock In: D.~Fleet, T.~Pajdla, B.~Schiele, T.~Tuytelaars (eds.) Computer
  Vision -- ECCV 2014, pp. 253--268. Springer International Publishing, Cham
  (2014)

\bibitem{lacoste2015global}
Lacoste-Julien, S., Jaggi, M.: On the global linear convergence of frank-wolfe
  optimization variants.
\newblock Advances in neural information processing systems \textbf{28} (2015)

\bibitem{lan2013complexity}
Lan, G.: The complexity of large-scale convex programming under a linear
  optimization oracle.
\newblock arXiv preprint arXiv:1309.5550  (2013)

\bibitem{lan2020first}
Lan, G.: First-order and stochastic optimization methods for machine learning,
  vol.~1.
\newblock Springer (2020)

\bibitem{levitin1966constrained}
Levitin, E.S., Polyak, B.T.: Constrained minimization methods.
\newblock USSR Computational mathematics and mathematical physics
  \textbf{6}(5), 1--50 (1966)

\bibitem{pedregosa2020linearly}
Pedregosa, F., Negiar, G., Askari, A., Jaggi, M.: Linearly convergent
  frank-wolfe with backtracking line-search.
\newblock In: International conference on artificial intelligence and
  statistics, pp. 1--10. PMLR (2020)

\bibitem{pokutta2024frank}
Pokutta, S.: The frank-wolfe algorithm: a short introduction.
\newblock Jahresbericht der Deutschen Mathematiker-Vereinigung \textbf{126}(1),
  3--35 (2024)

\bibitem{rockafellar1997convex}
Rockafellar, R.T.: Convex analysis, vol.~11.
\newblock Princeton university press (1997)

\bibitem{vegh2012strongly}
V{\'e}gh, L.A.: Strongly polynomial algorithm for a class of minimum-cost flow
  problems with separable convex objectives.
\newblock In: Proceedings of the forty-fourth annual ACM symposium on Theory of
  computing, pp. 27--40 (2012)

\end{thebibliography}


\end{document}